\documentclass[12pt]{amsart}

\usepackage{amssymb, amscd}
\usepackage{latexsym,epsfig}
\usepackage[all]{xy}
\usepackage{pst-all}
\numberwithin{equation}{section}
\def\today{\ifcase\month\or Jan\or Febr\or  Mar\or  Apr\or May\or Jun\or  Jul\or Aug\or  Sep\or  Oct\or Nov\or  Dec\or\fi \space\number\day, \number\year}

\newcommand{\Sy}{{\mathrm{Sym}}} 
\newcommand{\CC}{\mathbb C}

\newcommand{\FF}{\mathbb F}

\newcommand{\QQ}{\mathbb Q}
\newcommand{\RR}{\mathbb R}

\newcommand{\ZZ}{\mathbb Z}

\newcommand{\Sym}{\rm Sym}
\numberwithin{equation}{section}
\newtheorem{theorem}{Theorem}[section]
\newtheorem{lemma}[theorem]{Lemma}
\newtheorem{proposition}[theorem]{Proposition}
\newtheorem{corollary}[theorem]{Corollary}
\newtheorem{conjecture}[theorem]{Conjecture}

\newtheorem{definition-lemma}[theorem]{Definition-Lemma}
\theoremstyle{definition}

\theoremstyle{remark}
\newtheorem{remark}[theorem]{Remark}

\begin{document}

\title[]{On vector-valued Siegel modular forms \\
of degree $2$ and weight $(j,2)$}
\author{Fabien Cl\'ery}
\address{Department of Mathematics,
Loughborough University,
England}
\email{cleryfabien@gmail.com}

\author{Gerard van der Geer}
\address{Korteweg-de Vries Instituut, Universiteit van
Amsterdam, Postbus 94248,
1090 GE  Amsterdam, The Netherlands.}
\email{geer@science.uva.nl}


\maketitle
\centerline{with two appendices by \sc Ga\"etan Chenevier}
\begin{abstract}
We formulate a conjecture that describes the vector-valued
Siegel modular forms of degree~$2$ and level $2$
of weight ${\rm Sym}^j\otimes {\det}^2$  and provide 
some evidence for it.
We construct such modular forms of weight $(j,2)$ via covariants of binary sextics
and calculate their Fourier expansions illustrating the effectivity
of the approach via covariants.
Two appendices contain related results of Chenevier; in particular a proof of
the fact that every modular form of degree $2$ and level $2$ and weight $(j,1)$
vanishes.
\end{abstract}

\begin{section}{Introduction}\label{sec-intro}
The usual methods for determining the dimensions of spaces of 
Siegel modular forms do not work for low weights. 
For Siegel modular forms of degree $2$ 
this means that we do not have formulas for the dimensions of the 
spaces of Siegel modular forms of weight $(j,k)$, that is, 
corresponding to ${\rm Sym}^j \otimes \det{}^k$, in case $k<3$. 
In this paper we propose a 
description of the spaces of cusp forms of weight $(j,2)$
on the level~$2$ principal congruence subgroup 
$$
\Gamma_2[2]=\ker({\rm Sp}(4,{\ZZ}) \to {\rm Sp}(4,{\ZZ}/2{\ZZ}))
$$
of $\Gamma_2={\rm Sp}(4,{\ZZ})$
and we provide some evidence for this conjectural description.

Let $S_{j,k}(\Gamma_2[2])$ be the space of cusp forms 
of weight $(j,k)$, that is,
corresponding to the factor of automorphy 
${\rm Sym}^j(c\tau+d) \det(c\tau+d)^k$ on the group 
$ \Gamma_2[2]$.
Recall that the group ${\rm Sp}(4,{\ZZ}/2{\ZZ})$ is isomorphic to
the symmetric group $\mathfrak{S}_6$. We fix an explicit isomorphism
by identifying the symplectic lattice over ${\ZZ}/2{\ZZ}$ with the subspace
$\{ (a_1,\ldots,a_6) \in ({\ZZ}/2{\ZZ})^6: \sum a_i=0\}$ modulo the
diagonally embedded ${\ZZ}/2{\ZZ}$ with form $\sum_i a_ib_i$ as in
\cite[Section 2]{B-F-vdG1}; it is given explicitly on generators
of $\mathfrak{S}_6$ in \cite[Section 3, (3.2)]{C-vdG-G}. So
$\mathfrak{S}_6$
acts on the space of cusp forms $S_{j,k}(\Gamma_2[2])$ and this space 
thus decomposes into isotypical components for
the symmetric group $\mathfrak{S}_6$. The irreducible representations of $\mathfrak{S}_6$
correspond to the partitions of $6$ and we thus have for each such partition 
$\varpi$  
a subspace $S_{j,k}(\Gamma_2[2])^{s[\varpi]}$ 
of $S_{j,k}(\Gamma_2[2])$ where $\mathfrak{S}_6$ acts as $s[\varpi]$.
Note that the case $s[6]$ corresponds to cusp forms on ${\rm Sp}(4,{\ZZ})$,
while the case $s[1^6]$ corresponds to modular forms of weight $(j,k)$ on
${\rm Sp}(4,{\ZZ})$ with a quadratic character:
$$
S_{j,k}(\Gamma_2[2])^{s[1^6]} = S_{j,k}({\rm Sp}(4,{\ZZ}),\epsilon)
$$
with $\epsilon$ the unique quadratic character of ${\rm Sp}(4,{\ZZ})$.

\medskip
Before we formulate our conjecture we recall that the group 
${\rm SL}(2,{\ZZ}/2{\ZZ})\cong \mathfrak{S}_3$ acts on the space 
$S_k(\Gamma_1[2])$ of cusp forms
on the principal congruence subgroup of level $2$
$\Gamma_1[2]=\ker ({\rm SL}(2,{\ZZ})\to {\rm SL}(2,{\ZZ}/2{\ZZ}))$. 
We can thus decompose this space in isotypical components corresponding to the
irreducible representations of $\mathfrak{S}_3$. The map $f(z) \mapsto f(2z)$
defines an isomorphism $S_k(\Gamma_1[2]){\buildrel \sim \over \longrightarrow}
S_k(\Gamma_0(4))$ with $\Gamma_0(4)$ the usual congruence subgroup of 
$\Gamma_1={\rm SL}(2,{\ZZ})$. 
If we write its isotypical decomposition as
$$
S_k(\Gamma_1[2])= a_k\, s[3]+b_k \, s[2,1]+c_k \, s[1^3]\, ,
$$
then $\dim S_k(\Gamma_1)=a_k$, $\dim S_k(\Gamma_0(2))^{\rm new}=b_k-a_k$ and 
$\dim S_k(\Gamma_0(4))^{\rm new}= c_k$ with their generating series given by
$$
\sum a_k t^k= t^{12}/(1-t^4)(1-t^6), \quad \sum b_k t^k=t^8/(1-t^2)(1-t^6)\, , 
\quad \sum c_k t^k=t^6/(1-t^4)(1-t^6)\, .
$$ 
The Fricke involution $w_2: \tau \mapsto -1/2\tau$ defines an 
involution on $S_k(\Gamma_0(2))^{\rm new}$ and this space splits into 
eigenspaces $S_k^{\pm}(\Gamma_0(2))^{\rm new}$ and for $k>2$ we have
$$
\dim S_k^{+}(\Gamma_0(2))^{\rm new} - \dim S_k^{-}(\Gamma_0(2))^{\rm new}=
\begin{cases} -1 & k\equiv 2 \, \bmod\,  8 \\
0 & k \equiv 4,6 \, \bmod \, 8 \\
1 & k \equiv 0 \, \bmod \, 8\, . \\
\end{cases}
$$
 
\bigskip
We recall the notion of Yoshida type lifts.  
Yoshida lifts are explained in \cite{Y1}; see also \cite{Y2,weissauerbook,B-S,S-S}.
These are eigen forms associated to a pair of elliptic modular eigenforms
whose spinor L-function is a product of the twisted L-functions of
the elliptic modular cusp forms. In \cite{B-F-vdG1} a number of conjectures
on the existence of Yoshida lifts were made and these were proved by
R\"osner \cite{Roesner}. These conjectures deal with Siegel modular cusp 
forms of weight $(j,k)$ with $k\geq 3$. It can be extended to the case of
weight $(j,2)$.
We denote the subspace of $S_{j,2}(\Gamma_2[2])^{s[\varpi]}$ generated by Yoshida lifts 
by $YS_{j,2}^{s[\varpi]}$.

\begin{theorem}\label{YoshidaLifts} 
We have $YS_{j,2}^{s[w]} = 0$ unless we are in the following cases: \begin{itemize}\par \medskip

\item[(1)] $w=[1^6]$ and $YS_{j,2}^{s[w]}$ is generated by the $Y(f,g)$ with $f$ and $g$ 
eigen-newforms of level $\Gamma_0(2)$ of different sign. In this case  we have

$$\dim YS_{j,2}^{s[w]} \,= \,\dim \,S^+_{j+2}(\Gamma_0(2))^{\rm new} \otimes S^-_{j+2}(\Gamma_0(2))^{\rm new}.$$

\item[(2)] $w=[2,1^4]$ and $YS_{j,2}^{s[w]}$ is generated by the $Y(f,g)$ with $f$ and $g$ non proportional eigen-newforms on $\Gamma_0(4)$. The multiplicity $\mu(j)$ of $s[2,1^4]$ in $YS_{j,2}^{s[w]}$ is then
$$\mu(j)\, =\, \dim\, \Lambda^2 S_{j+2}(\Gamma_0(4))^{\rm new}.$$

\item[(3)] $w=[2^3]$ and $YS_{j,2}^{s[w]}$ is generated by the $Y(f,g)$ with $f$ and $g$ non proportional eigen-newforms on $\Gamma_0(2)$ with the same sign. The multiplicity $\nu(j)$ of $s[2^3]$  is

$$\nu(j) \,= \,\dim \Lambda^2 S^+_{j+2}(\Gamma_0(2))^{\rm new} \oplus \Lambda^2 S^-_{j+2}(\Gamma_0(2))^{\rm new}.$$

\end{itemize}
\end{theorem}

The proof of this theorem follows  from results of R\"osner and 
Weissauer, in a way very similar to  R\"osner's proof of the 
Bergstr\"om-Faber-van der Geer conjecture in weight $k\geq 3$ 
\cite[\S 5.5]{Roesner}. 
In the second appendix Chenevier explains how to derive it. 
\par \smallskip

We now formulate our conjecture.

\begin{conjecture} The space $S_{j,2}(\Gamma_2[2])$ is generated by
Yoshida type lifts. 
\end{conjecture}

Note that this implies that $S_{j,2}(\Gamma_2[2])^{s[\varpi]}=(0)$ 
unless $\varpi= [1^6], [2,1^4]$ or $[2^3]$. In particular, it
implies that $S_{2,j}(\Gamma_2)=(0)$. The evidence we have for the latter is
the following.

\begin{theorem}\label{vanishingSj2level1}
We have $\dim S_{j,2}(\Gamma_2)=0$ for $j\leq 52$.
\end{theorem}
\noindent
For $j\leq 20$ the vanishing of $S_{j,2}(\Gamma_2)$ 
was proved by Ibukiyama, Wakatsuki and Uchida 
\cite[Lemma 2.1]{I}, \cite{I-W},
and \cite{Uchida}.

\smallskip
The evidence we have for the $s[1^6]$-part of the conjecture is the following.

\begin{theorem}
The dimension of $S_{j,2}(\Gamma_2,\epsilon)$ is given by the coefficient of $t^j$
in the expansion of $t^{12}/(1-t^6)(1-t^8)(1-t^{12})$ for $j\leq 30$.
\end{theorem}
Modular forms in $S_{j,2}(\Gamma_2,\epsilon)$ can be constructed explicitly using
covariants as explained in \cite{C-F-vdG}. 
We prove this theorem  by constructing a basis of the space $S_{j,7}(\Gamma_2)$ using 
covariants of binary sextics (see \cite{C-F-vdG}) and then by checking
which forms are divisible by the cusp form $\chi_5 \in S_{0,5}(\Gamma_2,\epsilon)$.
We thus give generators for these spaces $S_{j,2}(\Gamma_2[2])^{s[1^6]}$ for $j\leq 30$
and we can calculate Hecke eigenvalues for these.
For $j>30$ this becomes quite laborious. 

\smallskip

For all irreducible representations 
we have the following vanishing result:
\begin{proposition}
For any $\varpi$ we have $\dim S_{j,2}(\Gamma_2[2])^{s[\varpi]}=0$ for $j<12$.
\end{proposition}

We end with some remarks on other `small' weights.
The vanishing of $S_{j,1}(\Gamma_2)$ follows from work of Skoruppa 
\cite{Sk}.
In an appendix to this paper besides providing 
a different proof of the vanishing of $S_{j,2}(\Gamma_2)$
for $j\leq 38$, Chenevier gives a proof of the 
vanishing of $S_{j,1}(\Gamma_2[2])$.  For $k=3$ one knows that $S_{j,3}(\Gamma_2)=(0)$
for $j<36$. But $S_{36,3}$ is $1$-dimensional and using covariants
we can construct a form of this weight in a relatively easy manner;
cf.\ the remarks in \cite[p.\ 207]{I-W} on the difficulty of constructing
such a form.
\medskip

\noindent
{\bf Acknowledgement.}
We thank Ga\"etan Chenevier warmly 
for asking the question on the existence of modular
forms of weight $(j,2)$ on the full group ${\rm Sp}(4,{\ZZ})$ and
for agreeing to add his results in the form of an appendix.
He also pointed out an error in an earlier version of our conjecture and
provided a proof of the extension of R\"osner's result on Yoshida lifts.
We thank Mirko R\"osner for correspondence.
We also thank the Max-Planck Institut f\"ur Mathematik 
in Bonn for excellent working conditions.
\end{section}
\begin{section}{Modular Forms of Degree Two}\label{MFDeg2}
For the definitions of Siegel modular forms and elementary properties we refer to
\cite{vdG1}.
We denote the Siegel upper half space of degree $g$ by
$\mathfrak{H}_g$. The Siegel 
modular group $\Gamma_g={\rm Sp}(2g,{\ZZ})$ 
acts on $\mathfrak{H}_g$ by fractional linear transformations
$$
\tau \mapsto (a\tau +b)(c\tau +d)^{-1} \qquad 
\text{\rm for $\left(\begin{matrix} 
a & b\cr c & d \cr\end{matrix} \right) \in {\rm Sp}(2g,{\ZZ})$
and $\tau \in \mathfrak{H}_g$.}
$$
If $\rho: {\rm GL}(g,{\CC}) \to {\rm GL}(W)$ is a finite-dimensional complex representation
then a holomorphic map $f: \mathfrak{H}_g \to W$ is called a Siegel modular form of weight $\rho$
if $f(a\tau +b)(c\tau +d)^{-1})=\rho(c\tau+d)f(\tau)$ 
for all $(a,b;c,d) \in \Gamma_g$.
The space of modular forms of weight $\rho$ is finite-dimensional and denoted by
$M_{\rho}(\Gamma_g)$.

If $g=2$ then an irreducible representation of ${\rm GL}(2,{\CC})$ is of the form
${\rm Sym}^j({\rm St}) \otimes {\det}({\rm St})^k$  with ${\rm St}$ 
the standard representation of ${\rm GL}(2,{\CC})$ 
for some $j \in {\ZZ}_{\geq 0}$ and $k \in {\ZZ}$. 

For $\rho={\rm Sym}^j({\rm St}) \otimes {\det}({\rm St})^k$ we denote $M_{\rho}$ by $M_{j,k}$
and we call $(j,k)$ the weight.
If $j=0$ we are dealing with scalar-valued modular forms. 
The space of Siegel modular forms of degree $2$ and weight $(j,k)$ 
is denoted by $M_{j,k}(\Gamma_2)$.

There is the Siegel operator $\Phi_g$ that maps Siegel modular forms of degree $g$ to Siegel modular forms of degree
$g-1$. The kernel of $\Phi_2$ in $M_{j,k}(\Gamma_2)$ is called the space of 
cusp forms of weight $(j,k)$ and denoted by $S_{j,k}(\Gamma_2)$.
Note that for $k=2$ we have $M_{j,2}(\Gamma_2)=S_{j,2}(\Gamma_2)$, see \cite[Lemma 2.1]{I}.

For a finite index subgroup $\Gamma$ of ${\rm Sp}(4,{\ZZ})$ we have similar notions. Here we
deal with the groups $\Gamma_2$ and $\Gamma_2[2]$. 
The quotient group $\Gamma_2/\Gamma_2[2]
\cong {\rm Sp}(4,{\ZZ}/2{\ZZ})$ 
is identified with the symmetric group $\mathfrak{S}_6$ as in the 
Introduction. This group acts in a natural
way on the space of cusp forms $S_{j,k}(\Gamma_2[2])$ and we can decompose this space in
isotypical components $S_{j,k}(\Gamma_2[2])^{s[\varpi]}$ corresponding to the irreducible
representations $s[\varpi]$ of $\mathfrak{S}_6$ which in turn correspond bijectively to the
partitions $\varpi$ of $6$. 

The ring $R$ 
of scalar-valued Siegel modular forms on $\Gamma_2$ was determined by Igusa
in the 1960s, see \cite{Igusa}. In the 1980s Tsushima gave in \cite{Tsushima}
formulas for the dimensions of the spaces of 
vector-valued  cusp forms on a subgroup between  $\Gamma_2[2]$ and $\Gamma_2$. 
Bergstr\"om extended this
to $\Gamma_2[2]$ with the action of $\mathfrak{S}_6$, see \cite{website}. 
We thus know the dimension
of $S_{j,k}(\Gamma_2[2])^{s[\varpi]}$ for all $j$ and $k\geq 3$.

The vector-valued modular forms of degree $2$ form a ring 
$M=\oplus_{j,k}  M_{j,k}(\Gamma_2)$. It is also a 
module over the ring $R$. For level $2$ similar things hold.

A vector-valued Siegel modular form $f$ of weight $(j,k)$ on $\Gamma_2$
has a Fourier-Jacobi expansion
$$
f(\tau)= \sum_{m\geq 0} \varphi_m(\tau_1,z) \, e^{2\pi i m \tau_2}
\qquad 
\text{ where $\tau= \left( \begin{matrix} \tau_1 & z \\
z & \tau_2 \\ \end{matrix} \right)$ }
$$
with $\varphi_m: \mathfrak{H}_1 \times {\CC} \to {\rm Sym}^j({\CC}^2)$
a holomorphic map that satisfies certain functional equations under the
action of the so-called Jacobi group ${\rm SL}(2,{\ZZ}) \ltimes {\ZZ}^2$.
This group is embedded in $\Gamma_2$ via
$$
\left( \begin{matrix} a & b \\ c & d \\ \end{matrix}\right) \mapsto 
\left( \begin{matrix} a & 0 & b & 0 \\
0 & 1 & 0 & 0 \\
c & 0 & d & 0 \\
0 & 0 & 0 & 1 \\ \end{matrix} \right)
, \qquad
(\lambda, \mu) \mapsto 
\left( \begin{matrix} 1 & 0 & 0 & \mu \\
\lambda  & 1 & \mu & 0 \\
0 & 0 & 1 & -\lambda \\
0 & 0 & 0 & 1 \\ \end{matrix} \right)
$$
with action
$$
\tau \mapsto \left( \begin{matrix} (a\tau_1+b)/(c\tau_1+d) & z/(c\tau_1+d) \\
z/(c\tau_1+d) & \tau_2-cz^2/(c\tau_1+d)\\ \end{matrix} \right)
$$
and
$$
\tau \mapsto \left( \begin{matrix} \tau_1 & z+\lambda \tau_1+\mu \\
z+\lambda \tau_1+\mu & \tau_2+ \lambda^2\tau_1+2\lambda z +\lambda \mu \\ \end{matrix} \right) \, .
$$
The fact that $f$ is a modular form of weight $(j,k)$ implies the corresponding functional
equations
$$
\varphi_m(\frac{a \tau_1+b}{c\tau_1+d}, \frac{z}{c\tau_1+d}) 
e^{-2\pi i m \frac{c z^2}{c\tau_1+d}}= {\rm Sym}^j
\left( \begin{matrix} c\tau_1+d & c z \\ 0 & 1 \\ \end{matrix} \right) (c\tau_1+d)^k 
\varphi_m(\tau_1,z)
$$
and
$$
\varphi_m(\tau_1,z+\lambda \tau_1+\mu) e^{2\pi i m (\lambda^2 \tau_1 +2\lambda z +\lambda \mu)} =
{\rm Sym}^j \left( \begin{matrix} 1 & -\lambda \\ 0 & 1 \\ \end{matrix} \right) \varphi_m(\tau_1,z) \, ,
$$
where we write $\varphi_m$ as the transpose of the row vector
$(\varphi_m^{(0)}, \ldots, \varphi_m^{(j)})$.
\begin{corollary}
If $f \in M_{j,k}(\Gamma_2)$ (resp.\ $f \in S_{j,k}(\Gamma_2)$) then
the last coordinate $\varphi_m^{(j)}$ of the coefficient $\varphi_m$ of $e^{2\pi i m \tau_2}$
in the Fourier-Jacobi expansion  of $f$
is a Jacobi form (resp.\ Jacobi cusp form) of weight $k$ and index $m$.
\end{corollary}
We note that  Jacobi cusp forms of weight $2$ and index $m$ are zero for $m< 37$,
see \cite[pp.\ 117-120]{E-Z}. This imposes strong conditions on forms of weight
$(j,2)$ on $\Gamma_2$.

\medskip
Since we shall compute the action of Hecke operators later
we now describe formulas for the action of 
Hecke operators on forms  on $\Gamma_2$. For forms without
character we refer to \cite[Appendix]{C-vdG}, so we deal with the case of
forms with a character. For $f \in M_{j,k}(\Gamma_2,\epsilon)$
we write its Fourier expansion as
$$
f(\tau)= \sum_{n\geq 0} a(n) \, e^{\pi i {\rm Tr}(n\tau)} \, ,
$$
where $n$ runs over the positive semi-definite half-integral symmetric matrices.
We will write $[n_1,n_2,n_3]$ for $\left( \begin{smallmatrix}
n_1 & n_2/2 \\ n_2/2 & n_3 \end{smallmatrix} \right)$.
For an odd prime $p$ we denote by  $T_p$ the Hecke operator for $\Gamma_2$ 
at $p$. Then we write the transform of $f$ under $T_p$ as
$$
T_p(f)(\tau)= \sum_{n\geq 0} a_p(n) \, e^{\pi i {\rm Tr} (n\tau)}\, .
$$
Here for $p \not\equiv 1 \bmod\,  3$ the coefficient $a_p([1,1,1])$ is
given by  $a([p,p,p])$, 
and for $p=3$ by
$$
a([3,3,3])-3^{k-2} 
\Sy^j
 \left(
\begin{smallmatrix}
3 & -1  \\
0 & 1 \\
\end{smallmatrix}
\right) 
a([1,3,3])  \, ,
$$
while for $p \equiv 1 \bmod \, 3$ by
$$
a([p,p,p]) 
+
p^{k-2} 
\sum_{i=1}^{2}
(-1)^{m_i} 
\Sy^j
 \left(
\begin{smallmatrix}
p & -m_i  \\
0 & 1 \\
\end{smallmatrix}
\right) 
a([\frac{1+m_i+m_i^2}{p},1+2m_i, p])  \, ,
$$
where in the latter case $m_1$ and $m_2$ are 
the two roots of the polynomial $1+X+X^2$ over $\FF_p$, which we view  
here 
as the set  $\left\{0,\ldots,p-1\right\}$.

Similarly, the coefficient $a_{p^2}([1,1,1])$ of the transform of $f$ 
under the Hecke operator $T_{p^2}$ is given for $p \not\equiv 1 \bmod\, 3$
by $a([p^2,p^2,p^2])$, and for $p=3$ by 
$$
a([9,9,9])-3^{k-2} 
\Sy^j
 \left(
\begin{smallmatrix}
3 & -1  \\
0 & 1 \\
\end{smallmatrix}
\right) 
a([3,9,9]) \, .
$$

As an example, consider the modular form
$$
\chi_5 \in S_{0,5}(\Gamma_2,\epsilon)\,,
$$ 
the product of the ten even theta characteristics 
and the square root of Igusa's cusp form $\chi_{10}$, 
that will play an important role
in this paper. It provides a check on these formulas for the Hecke
operators.  Indeed, one knows
$$
\lambda_p(\chi_5)=p^3+a_p(f)+p^4, \qquad 
\lambda_{p^2}(\chi_5)=\lambda_p(\chi_5)^2-(p^4+p^3)\lambda_p(\chi_5)+p^{8}\, ,
$$
where $f=q-8\, q^2+12\, q^3+64\, q^4-210\, q^5 +\cdots $ is the 
normalized Hecke eigenform in $S_{8}^+(\Gamma_0(2))^{\text{new}}$, which
illustrates that $\chi_5$ is a Saito-Kurokawa lift. One can check
that the above formulas agree
with this.

\end{section}
\begin{section}{Restriction to the diagonal}
In order to put restrictions on the existence of Siegel modular forms 
we restrict these to the `diagonal' given by the embedding
$$
i: \mathfrak{H}_1\times \mathfrak{H}_1 \to \mathfrak{H}_2,\qquad
(z_1,z_2) \mapsto \left(\begin{matrix}z_1 & 0 \cr 0 & z_2\cr\end{matrix}\right)\, .
$$
The stabilizer of $i(\mathfrak{H}_1\times \mathfrak{H}_1)$ 
in ${\rm Sp}(4,{\RR})$ is an extension by ${\ZZ}/2{\ZZ}$ of the image of 
${\rm SL}(2,{\RR})^2$ under the embedding
$$
(\left(
\begin{smallmatrix}
a_1 & b_1 \\ c_1 & d_1
\end{smallmatrix}
\right)
,
\left(
\begin{smallmatrix}
a_2 & b_2 \\ c_2 & d_2
\end{smallmatrix}
\right)
) \, \mapsto \, 
\left(
\begin{smallmatrix}
a_1 & 0 & b_1 & 0\\
0 & a_2 & 0 & b_2 \\
c_1 & 0 & d_1 & 0\\
0 & c_2 & 0 & d_2 \\
\end{smallmatrix}
\right)
$$
The extension by ${\ZZ}/2{\ZZ}$ 
corresponds to the involution that interchanges 
$\tau_1$ and $\tau_2$ 
in 
$$
\tau =\left(\begin{smallmatrix}
\tau_1 & \tau_{12}\cr \tau_{12} & \tau_2 \cr \end{smallmatrix} \right) \in \mathfrak{H}_2
$$ 
(and  $z_1$ and $z_2$ on $\mathfrak{H}_1^2$). This corresponds to the element 
$\iota=\left(\begin{smallmatrix} a & 0 \\ 0 & d \\ \end{smallmatrix}\right)$ in
$\Gamma_2$ with 
$a=d= \left(\begin{smallmatrix} 0 & 1 \\ 1 & 0 \\ \end{smallmatrix}\right)$.
The stabilizer inside $\Gamma_2$ (resp.\ inside $\Gamma_2[2]$) is an
extension by ${\ZZ}/2{\ZZ}$  of ${\rm SL}(2,{\ZZ})\times {\rm SL}(2,{\ZZ})$ 
(resp.\ of $\Gamma_1[2]\times \Gamma_1[2]$).

If $F=(F_0,\ldots,F_{j})^t$ is a Siegel modular form
of weight $(j,k)$ of level $2$, then its pullback
under $i$
to $\mathfrak{H}_1\times \mathfrak{H}_1$ 
gives rise to an element of $(f_0,\ldots, f_j)^t$
with $f_l \in M_{j+k-l}(\Gamma_1[2])\otimes M_{k+l}(\Gamma_1[2])$.

By restricting a cusp form of level $1$  
we get cusp forms of level $1$.
The action of  $\iota$ is given by a map
$$
S_{j+k-i}(\Gamma_1[2])\otimes S_{k+i}(\Gamma_1[2]) \to
S_{k+i}(\Gamma_1[2])\otimes S_{j+k-i}(\Gamma_1[2]), \qquad
a\otimes b \mapsto (-1)^k b\otimes a
$$ 
for $\Gamma_1[2]$ and a similar one for $\Gamma_1$. 
So for a form of level $1$ without character 
we loose no information by looking at
$$
\bigoplus_{i=0}^{j/2-1} S_{j+k-i}(\Gamma_1)\otimes S_{k+i}(\Gamma_1) 
 \bigoplus \begin{cases}
{\wedge}^2 S_{j/2+k}(\Gamma_1) & \text{for $k$ odd} \\
{\rm Sym}^2 S_{j/2+k}(\Gamma_1) & \text{for $k$ even.} \\
\end{cases} 
$$

By multiplying with $\chi_5$ we get an injective map 
 $S_{j,2}(\Gamma_2,\epsilon) \to S_{j,7}(\Gamma_2)$. 
The generating series for the dimensions is
$$
\sum_{j=2}^{\infty} \dim S_{j,7}(\Gamma_2) \, t^j =
{t^{12} \over (1-t)(1-t^3)(1-t^4)(1-t^6)} \, .
$$
We observe that our conjecture on $S_{j,2}(\Gamma_2,\epsilon)$ implies that 
$$
\dim S_{j,7}(\Gamma_2) -\dim S_{j,2}(\Gamma_2,\epsilon) =
\sum_{i=0}^{j/2-1} \dim S_{j+7-i}(\Gamma_1) \dim S_{7+i}(\Gamma_1)
+ \dim \wedge^2 S_{j/2+7}(\Gamma_1) \, ,
$$
or equivalently, that the restriction $\rho$ to the diagonal
fits in  an exact sequence
$$
0 \to S_{j,2}(\Gamma_2,\epsilon) {\buildrel \cdot \chi_5 \over\longrightarrow} 
S_{j,7}(\Gamma_2)
{\buildrel \rho \over \longrightarrow} \oplus_{i=0}^{j/2-1}S_{j+7-i}(\Gamma_1)\otimes S_{7+i}(\Gamma_1)
\oplus \wedge^2 S_{j/2+7}(\Gamma_1)
\to 0 \, .
$$

\medskip
If $F\in S_{j,k}(\Gamma_2,\epsilon)$ 
 we find by using $\iota$ that
we can restrict to
$$
\bigoplus_{i=0}^{j/2-1} S_{j+k-i}(\Gamma_1[2])^{s[1^3]}
\otimes S_{k+i}(\Gamma_1[2])^{s[1^3]} 
\bigoplus
{\rm Sym}^2(S_{j/2+k}(\Gamma_1[2])^{s[1^3]})\, .
$$
Indeed, the group $\mathfrak{S}_3={\rm SL}(2,{\ZZ}/2{\ZZ})$ acts on $S_k(\Gamma_1[2])$
and  for  a form on $\Gamma_2$ with a character
the components $f_i, f_i'$ of the
restriction to $i(\mathfrak{H}_1\times \mathfrak{H}_1)$ are modular forms
on $\Gamma_1[2]$ with a character, i.e.\ they lie in the $s[1^3]$-isotypical
part of $S_k(\Gamma_1[2])$. The module $\oplus_k S_k(\Gamma_1[2])^{s[1^3]}$
is a module over the ring ${\CC}[e_4,e_6]$ 
of modular forms on $\Gamma_1$ and is generated
by the cusp form $\delta=\eta^{12}$, a square root of 
$\Delta\in S_{12}(\Gamma_1)$.

The generating series for the dimensions is now
$$
\sum_{j=2}^{\infty} \dim S_{j,7}(\Gamma_2,\epsilon)\, t^j =
{t^6 \over (1-t)(1-t^3)(1-t^4)(1-t^6)}\, .
$$
Conjecturally we now find an exact sequence
$$
\begin{aligned}
0 \to S_{j,7}(\Gamma_2,\epsilon) {\buildrel \rho \over \longrightarrow} 
\oplus_{i=0}^{j/2-1} S_{j+7-i}(\Gamma_0(4))^{n} 
\otimes S_{7+i}(\Gamma_0(4))^{n}  \, 
\oplus {\rm Sym}^2 (S_{j/2+7}(\Gamma_0(4))^{n}) &\\
 \to K \to 0\, , &\\
\end{aligned}
$$
where $S_k(\Gamma_0(4))^{n}=S_k(\Gamma_0(4))^{\rm new}$  and
the dimension of the cokernel $K$ is now predicted by minus 
the extrapolation to $k=2$ of the algorithms used in \cite{B-F-vdG1} to 
calculate the dimension of $S_{j,k}(\Gamma_2)$ and which give negative numbers
here. 
\end{section}
\begin{section}{Constructing Modular Forms Using Covariants}\label{covariants}
In the paper \cite{C-F-vdG} we explained how to use invariant theory to construct Siegel modular
forms. In this paper we shall make extensive use of the procedure.

Let $V$ be the standard representation space of ${\rm GL}(2,{\CC})$
with basis $x_1,x_2$. We consider the space ${\rm Sym}^6(V)$ of
binary sextics, where we write an element as
$$
f= \sum_{i=0}^6 a_i \binom{6}{i} x_1^{6-i}x_2^i \, .
$$
Sometimes we call this expression the universal binary sextic.
For a description of invariants and covariants for the
action of ${\rm GL}(2,{\CC})$ we refer to  \cite[Section 3]{C-F-vdG}.
An invariant can be viewed as a polynomial in the coefficients $a_i$
that is invariant under the action of ${\rm SL}(2,{\CC})$, while a
 covariant of degree $(a,b)$ can be viewed as a form of degree $a$
in the $a_i$ and degree $b$ in $x_1,x_2$. If $A[\lambda_1,\lambda_2]$
is an irreducible representation of highest weight $(\lambda_1 \geq \lambda_2)$
of ${\rm GL}(2,{\CC})$ embedded equivariantly
in ${\rm Sym}^d({\rm Sym}^6(V))$ this defines a covariant of degree 
$(d,\lambda_1-\lambda_2)$ and it is unique up to a multiplicative non-zero 
constant. 

We denote the ring of covariants by ${\mathcal C}$. 
Clebsch and others constructed in the
19th century generators for this ring. There are $26$ generators,
$5$ invariants and $21$ covariants, satisfying many relations. 
They can be found in the book of
Grace and Young \cite[p.\ 156]{G-Y}. For the convenience of the reader we
reproduce these here. In the following table $C_{a,b}$
denotes a generator of degree $(a,b)$.

\begin{footnotesize}
\smallskip
\vbox{
\bigskip\centerline{\def\quad{\hskip 0.6em\relax}
\def\quod{\hskip 0.5em\relax }
\vbox{\offinterlineskip
\hrule
\halign{&\vrule#&\strut\quod\hfil#\quad\cr
height2pt&\omit &&\omit &&\omit &&\omit &&\omit &&\omit &&\omit &&\omit &\cr
&$a \backslash b$&&
$0$  && $2$ && $4$ && $6$ && $8$ && $10$ && $12$ &\cr
\noalign{\hrule}
& 1 &&   &&  &&  && $C_{1,6}$ &&  &&  &&  &\cr
\noalign{\hrule}
& 2 && $C_{2,0}$  &&  && $C_{2,4}$ &&  && $C_{2,8}$ &&  &&  &\cr
\noalign{\hrule}
& 3 &&   && $C_{3,2}$ &&  && $C_{3,6}$ && $C_{3,8}$ &&  && $C_{3,12}$ &\cr
\noalign{\hrule}
& 4 && $C_{4,0}$  &&  && $C_{4,4}$ && $C_{4,6}$ &&  && $C_{4,10}$ &&  &\cr
\noalign{\hrule}
& 5 &&   && $C_{5,2}$ && $C_{5,4}$ &&  && $C_{5,8}$ &&  &&  &\cr
\noalign{\hrule}
& 6 && $C_{6,0}$  &&  &&  && $C_{6,6}^{(1)}$ &&  &&  &&  &\cr
&  &&   &&  &&  && $C_{6,6}^{(2)}$ &&  &&  &&  &\cr
\noalign{\hrule}
& 7 &&   && $C_{7,2}$ && $C_{7,4}$ &&  &&  &&  &&  &\cr
\noalign{\hrule}
& 8 &&   && $C_{8,2}$ &&  &&  &&  &&  &&  &\cr
\noalign{\hrule}
& 9 &&   &&  && $C_{9,4}$ &&  &&  &&  &&  &\cr
\noalign{\hrule}
& 10 && $C_{10,0}$  && $C_{10,2}$ &&  &&  &&  &&  &&  &\cr
\noalign{\hrule}
& 12 &&   && $C_{12,2}$ &&  &&  &&  &&  &&  &\cr
\noalign{\hrule}
& 15 && $C_{15,0}$  &&  &&  &&  &&  &&  &&  &\cr
} \hrule}
}}
\end{footnotesize}

A theorem of Gordan says that all
these covariants can be constructed explicitly by using so-called
transvectants from the universal binary sextic. 
If  ${\Sym}^m(V)$ denotes 
the space of binary quantics of degree $m$ then we define the $k$th
transvectant as follows. It is a map 
${\Sym}^m(V) \times {\Sym}^n(V) \to {\Sym}^{m+n-2k}(V)$
that sends a pair $(f,g)$ to
$$
(f,g)_k=\frac{(m-k)!(n-k)!}{m!n!}\sum_{j=0}^k (-1)^j
\binom{k}{j}
\frac{\partial ^k f}{\partial x_1^{k-j}\partial x_2^j}
\frac{\partial ^k g}{\partial x_1^{j}\partial x_2^{k-j}}\, .
$$
When $k=1$, we omit the index: $(f,g)=(f,g)_1$. 
The next table gives the construction
of the covariants in the preceding table.

\begin{footnotesize}
\smallskip
\vbox{
\bigskip\centerline{\def\quad{\hskip 0.6em\relax}
\def\quod{\hskip 0.5em\relax }
\vbox{\offinterlineskip
\hrule
\halign{&\vrule#&\strut\quod\hfil#\quad\cr
height2pt&\omit &&\omit &&\omit &&\omit &&\omit &\cr
& 1 && $C_{1,6}=f$  && && && &\cr
\noalign{\hrule}
& 2  && $C_{2,0}=(f,f)_6$ && $C_{2,4}=(f,f)_4$  && $C_{2,8}=(f,f)_2$ && &\cr
\noalign{\hrule}
& 3  && $C_{3,2}=(C_{1,6},C_{2,4})_4$ && $C_{3,6}=(f,C_{2,4})_2$ && $C_{3,8}=(f,C_{2,4})$ && $C_{3,12}=(f,C_{2,8})$ &\cr
\noalign{\hrule}
& 4  && $C_{4,0}=(C_{2,4},C_{2,4})_4$ && $C_{4,4}=(f,C_{3,2})_2$ &&  $C_{4,6}=(f,C_{3,2})$ && $C_{4,10}=(C_{2,8},C_{2,4})$ &\cr
\noalign{\hrule}
& 5  && $C_{5,2}=(C_{2,4},C_{3,2})_2$ && $C_{5,4}=(C_{2,4},C_{3,2})$ && $C_{5,8}=(C_{2,8},C_{3,2})$ && &\cr
\noalign{\hrule}
& 6  && $C_{6,0}=(C_{3,2},C_{3,2})_2$ && $C_{6,6}^{(1)}=(C_{3,6},C_{3,2})$ && $C_{6,6}^{(2)}=(C_{3,8},C_{3,2})_2$ && &\cr
\noalign{\hrule}
& 7  && $C_{7,2}=(f,C_{3,2}^2)_4$ && $C_{7,4}=(f,C_{3,2}^2)_3$ && && &\cr
\noalign{\hrule}
& 8  && $C_{8,2}=(C_{2,4},C_{3,2}^2)_3$ && && && &\cr
\noalign{\hrule}
& 9  && $C_{9,4}=(C_{3,8},C_{3,2}^2)_4$ && && && &\cr
\noalign{\hrule}
& 10  && $C_{10,0}=(f,C_{3,2}^3)_6$ && $C_{10,2}=(f,C_{3,2}^3)_5$ && && &\cr
\noalign{\hrule}
& 12  && $C_{12,2}=(C_{3,8},C_{3,2}^3)_6$ && && && &\cr
\noalign{\hrule}
& 15  && $C_{15,0}=(C_{3,8},C_{3,2}^4)_8$ && && && &\cr
} \hrule}
}}
\end{footnotesize}

\bigskip

Let $M$ be the ring of vector-valued Siegel modular forms of degree $2$.
It is a module over the ring $R$ of scalar-valued Siegel modular forms of
degree $2$. In \cite{C-F-vdG} we defined maps
$$
M \longrightarrow {\mathcal C} {\buildrel \nu \over  \longrightarrow}
 M_{\chi_{10}}\, ,
$$
where $M_{\chi_{10}}$ is the localization of $M$ at $\chi_{10}$.
A modular form of weight $(j,k)$ maps to a covariant of degree $(j/2+k,j)$
and a covariant of degree $(a,b)$ is sent to a meromorphic modular form 
of weight $(b,a-b/2)$.
Under the map $\nu$ the universal binary sextic $f$ is mapped to 
$\chi_{6,3}/\chi_5$ of weight $(6,-2)$. 
Here $\chi_{6,3}$ is a holomorphic form in $S_{6,3}(\Gamma_2,\epsilon)$.
The beginning of its Fourier expansion is given in \cite{C-F-vdG}.

In practice  instead of $\nu$ often we use a slightly modified map 
$$
\mu: {\mathcal C} \longrightarrow M\oplus M_{\epsilon}\, ,
$$
where $M_{\epsilon}= \oplus M_{j,k}(\Gamma_2,\epsilon)$, is 
the $R$-module of modular
forms with a character. Under $\mu$
the universal sextic $f$ is mapped to $\chi_{6,3}$. 
Then a covariant maps of degree $(a,b)$ maps 
to a holomorphic Siegel modular form of weight $(b,6a-b/2)$ and character
$\epsilon^a$.

\begin{remark}\label{orderinfinity}
Since $\chi_{6,3}$ vanishes simply at infinity the definition of $\mu$ implies that
the image under $\mu$ of a covariant of degree $(a,b)$ vanishes at infinity with
order $\geq a$. Recall that the order of vanishing of $\chi_5$ at infinity is $1$. 
\end{remark}

For example, Igusa's generators $E_4,E_6,\chi_{10}, \chi_{12}$ and $\chi_{35}$
of $R$ 
are up to a non-zero multiplicative constant obtained as
$$
\begin{aligned}
&E_4=\mu(75\, C_{4,0} -8\, C_{2,0}^2)/\chi_{10}^2, \quad
E_6= \mu(224\, C_{2,0}^3 -1425\, C_{2,0}C_{4,0} -1125 \, C_{6,0})/\chi_{10}^3 \\
&\chi_{10}^6 = \mu(C_{\chi_{10}}), \quad
\chi_{12}=  \mu(C_{2,0}), \quad
\chi_{35}= \mu(C_{15,0})/\chi_5^{11}, \\
\end{aligned}
$$
with $C_{\chi_{10}}$, up to a multiplicative constant equal to the discriminant,
given by
$$
768\, C_{2,0}^5-7625 \, 
C_{4,0}C_{2,0}^3-1875\left(7 \, C_{6,0}C_{2,0}^2 -
10 \, C_{4,0}^2C_{2,0} 
-30  \, C_{6,0}C_{4,0}- 13860 \, C_{10,0}\right) \, . 
$$

\begin{remark}
The first scalar-valued cusp form on $\Gamma_2$ with character is of weight 
$30$ and can be obtained by dividing $\mu(C_{15,0})$ by $\chi_5^{12}$.
Note that we have 
$
M_{j,k}(\Gamma_2,\epsilon)=S_{j,k}(\Gamma_2,\epsilon).
$
(see \cite[p.\ 198]{I-W}). 
\end{remark}

\end{section}
\begin{section}{Cusp forms of weight $(j,2)$ on $\Gamma_2$ with a character}

Our conjecture says that $S_{j,2}(\Gamma_2,\epsilon)=(0)$ for $j<12$.
We begin by showing this.

\begin{lemma}
For $j=0, 2, 4, 6, 8$ and $10$, we have
$
S_{j,2}(\Gamma_2,\epsilon)=(0).
$
\end{lemma}
\begin{proof}
We know that $\dim S_{2j,4}(\Gamma_2)=0$ for $j=0,2,4,6,8,10$.
Assume that for one of these values of $j$ 
there is a non-zero element $f\in S_{j,2}(\Gamma_2,\epsilon)$.
Then ${\Sym}^2(f)\in S_{2j,4}(\Gamma_2)=(0)$ must be zero. 
Using the fact that the ring of holomorphic
functions on $\mathfrak{H}_2$ is an integral domain, we get a contradiction.
\end{proof}

The first case where $S_{j,2}(\Gamma_2,\epsilon)$ is predicted to be non-zero
is for $j=12$. In the paper \cite{C-F-vdG} we constructed a modular form
$\chi_{12,2}$ in this space using the covariant $C_{3,12}$ 
associated to $A[15,3]$
occurring in ${\rm Sym}^3({\rm Sym}^6(V))$, after dividing by the
cusp form $\chi_{10}$. 
Its Fourier expansion starts with
$$
\chi_{12,2}(\tau)=
\left(
\begin{smallmatrix}
0\\
0\\
0\\
2(R-R^{-1})\\
9(R+R^{-1})\\
12(R-R^{-1})\\
0\\
-12(R-R^{-1})\\
-9(R+R^{-1})\\
-2(R-R^{-1})\\
0\\
0\\
0
\end{smallmatrix}
\right)Q_1Q_2+\cdots \, ,
$$
where $Q_1= e^{\pi i \tau_1}$, $Q_2=e^{\pi i \tau_2}$ 
and
$R=e^{\pi i \tau_{12}}$
for $\tau=
\left(\begin{smallmatrix} \tau_1 & \tau_{12} \\ \tau_{12} & \tau_2\\ 
\end{smallmatrix} \right)$. 
By multiplication by $\chi_{6,3}$
we get an injective map $S_{12,2}(\Gamma_2,\epsilon) \to S_{18,5}(\Gamma_2)$
and this latter space is $1$-dimensional.

\begin{corollary}
We have $\dim S_{12,2}(\Gamma_2,\epsilon)=1$ and it is generated by $\chi_{12,2}$.
\end{corollary}

We compute a few Hecke eigenvalues as described in Section \ref{MFDeg2}.
To compute these Hecke eigenvalues, we used the following 
Fourier coefficients:

\begin{footnotesize}
\begin{align*}
a([1,1,1])^t&=[0, 0, 0, 2, 9, 12, 0, -12, -9, -2, 0, 0, 0]\\
a([1,3,3])^t&=[0, 0, 0, -2, -27, -156, -504, -996, -1233, -934, -396, -72, 0]\\
a([3,3,3])^t&=[0, 216, 1188, 258, -7749, -12708, 0, 12708, 7749, -258, -1188, -216, 0]\\
a([5,5,5])^t&=[0, 0, 0, -106920, -481140, -641520, 0, 641520, 481140, 106920, 0, 0, 0]\\
a([1,5,7])^t&=[0, 0, 0, 2, 45, 444, 2520, 9060, 21375, 33046, 32220, 17928, 4320]\\
a([7,7,7])^t&=[0, -8208, -45144, -542204, -2101338, -2711496, 0, 2711496, 2101338, 542204, 45144, 8208, 0]\\
a([3,9,7])^t&=[0, -72, -1188, -8854, -39339, -115764, -236880, -343884,\\
& \qquad \qquad -354141, -253514, -120132, -33912, -4320]\\
\end{align*}
\end{footnotesize}
These and a few more (too big to be written here) 
yield the following eigenvalues:

\begin{footnotesize}
\smallskip
\vbox{
\bigskip\centerline{\def\quad{\hskip 0.6em\relax}
\def\quod{\hskip 0.5em\relax }
\vbox{\offinterlineskip
\hrule
\halign{&\vrule#&\strut\quod\hfil#\quad\cr
height2pt&\omit&&\omit &&\omit &&\omit && \omit && \omit &&\omit &\cr
& $p$ && $3$ && $5$ && $7$ && $11$ && $13$ && $17$ & \cr
\noalign{\hrule}
height2pt&\omit&&\omit &&\omit &&\omit && \omit && \omit &&\omit &\cr
&$\lambda_p$ && $-600$ && $-53460$ && $-369200$ && $4084344$ && $-2845700$ &&
$131681700$ &\cr
} \hrule}
}}
\end{footnotesize}

\noindent
together with $\lambda_9=-1090791$. 
We find that for the operator $T_p$
these are indeed of the form $\lambda_p(f^{+})+\lambda_p(f^{-})$ with $f^{\pm}$
generators of $S^{\pm}_{14}(\Gamma_0(2))^{\rm new}$, with
\begin{align*}
f^{+}(\tau)&=q - 64\, q^2 - 1836\, q^3 + 4096\, q^4 + 3990\, q^5 + 
117504\, q^6 +\cdots \\
f^{-}(\tau)&=q + 64\, q^2 + 1236\, q^3 + 4096\, q^4 - 57450\, q^5 + 
79104\, q^6 +\cdots \, ,
\end{align*}
 while for $T_{p^2}$
we find
$\lambda_{p}(f^{+})^2+\lambda_p(f^{+})\lambda_{p}(f^{-})+
\lambda_p(f^{-})^2-2\, (p+1)p^j$.
This fits with being a Yoshida lift. 

\begin{lemma}
We have $ S_{14,2}(\Gamma_2,\epsilon)=(0)=S_{16,2}(\Gamma_2,\epsilon)$.
\end{lemma}
\begin{proof}
To see that $S_{14,2}(\Gamma_2,\epsilon)=(0)$ we multiply a form in 
$S_{14,2}(\Gamma_2,\epsilon)$ with $\chi_5$ and we end up in
$S_{14,7}(\Gamma_2)$ and this space is generated by a form associated to $C_{1,6}C_{3,8}$ after
division by $\chi_5^2$. 
Restricting this form $\chi_{14,7}$ to $\mathfrak{H}_1^2$ 
gives $\sum_{i=0}^{7}(f_i\otimes f_i')$
and only the term $f_5 \otimes f_5'$ in 
$S_{16}(\Gamma_1)\otimes S_{12}(\Gamma_1)$
can be non-zero and it is equal to $56\, e_4\Delta\otimes \Delta$.
So it does not vanish along $\mathfrak{H}_1\times \mathfrak{H}_1$, hence 
$\chi_{14,7}$ is not divisible by $\chi_5$. We conclude $S_{14,2}(\Gamma_2[2],\epsilon)=(0)$.

For $S_{16,2}(\Gamma_2,\epsilon)$ we multiply by $\chi_{6,3}$ and land in $S_{22,5}(\Gamma_2)$
and this space is zero.
\end{proof}
\bigskip

Next we deal with the case of weight $(18,2)$. 

\begin{proposition}
The space $S_{18,2}(\Gamma_2,\epsilon)$ has dimension $1$.
\end{proposition}
\begin{proof}
First we construct a non-zero element in this space by using the
covariant 
$$
C=135\,C_{1,6}^{2}C_{4,6}+56\,C_{1,6}C_{2,0}C_{3,12}-270\,C_{2,8}C_{4,10}-930\,C_{3,6}C_{3,12}.
$$
It occurs in $\Sym^6(\Sym^6(V))$ 
and provides a cusp form, $F_C$, of weight $(18,27)$ on $\Gamma_2$.
The order of vanishing
of $F_C$ along $\mathfrak{H}_1\times \mathfrak{H}_1$ is $5$, 
so we can divide it by $\chi_{5}^5$ and we get
by Remark \ref{orderinfinity} 
a cusp form, denoted $\chi_{18,2}$, of weight $(18,2)$ on $\Gamma_2$ with character.

Again we multiply by $\chi_{5}$ and
land in $S_{18,7}(\Gamma_2)$. This space is $2$-dimensional and we can 
construct a basis 
using the following covariants
$$
C_1=C_{3,12}(8\, C_{1,6}C_{2,0}-75\, C_{3,6})\quad
\text{ and} \quad 
C_2= C_{1,6}^2 C_{4,6}-2\, C_{2,8}C_{4,10}-3\, C_{3,6}C_{3,12}.
$$
They occur in $\Sym^6(\Sym^6(V))$ and provide two cusp forms, $F_{C_i}$, of 
weight $(18,27)$ on $\Gamma_2$. Each cusp form $F_{C_i}$ vanishes with 
order $4$ along $\mathfrak{H}_1\times \mathfrak{H}_1$,
so we can divide it by $\chi_{5}^4$ and get a cusp form, 
$\chi_{18,7}^{(i)}$, of weight $(18,7)$ on $\Gamma_2$.
The cusp forms $\chi_{18,7}^{(1)}$ and $\chi_{18,7}^{(2)}$ are 
$\CC$-linearly independent as can be read off from 
the first terms of their Fourier expansions
and the pullbacks to $\mathfrak{H}_1\times \mathfrak{H}_1$ are of the
form $\sum_{r=0}^9 f_r \otimes f_r'$ with only non-zero terms for $r=5$
and these are $216\, 
e_4^2\Delta \otimes \Delta$ and $48\, e_4^2\Delta \otimes \Delta$.
Up to a non-zero scalar there is only one non-trivial linear combination, 
that vanishes along 
$\mathfrak{H}_1\times \mathfrak{H}_1$ and that gives a non-zero form in $S_{18,2}(\Gamma_2,\epsilon)$ after division by $\chi_5$.  
\end{proof}

\begin{proposition}
The space $S_{20,2}(\Gamma_2,\epsilon)$ has dimension $1$.
\end{proposition}
\begin{proof}
We construct a non-zero form in this space by taking the covariant
$$
\begin{aligned}
C=&
224\,C_{1,6}^{2}C_{5,8}+312\,C_{1,6}C_{2,4}C_{4,10}-560\,C_{1,6}C_{2,8}C_{4,6}\\
&-108\,C_{1,6}C_{3,6}C_{3,8}+728\,C_{2,0}C_{2,8}C_{3,12}-1235\,C_{2,4}^{2}C_{3,12}.
\end{aligned}
$$
occurring in ${\rm Sym}^7({\rm Sym}^6(V))$
and providing a cusp form, $F_C$, of weight $(20,32)$ on $\Gamma_2$
with character.
The order of vanishing
of $F_C$ along $\mathfrak{H}_1\times \mathfrak{H}_1$ is 
$6$, so we can divide it by $\chi_{5}^6$ and we get
a cusp form, $\chi_{20,2}$, of weight $(20,2)$ on $\Gamma_2$ with character.

In a similar way we construct a basis of the space $S_{20,7}(\Gamma_2)$ by
taking the covariants
$$
\begin{aligned}
C_1&=
480\,C_{1,6}^{2}C_{5,8}-180\,C_{1,6}C_{3,6}C_{3,8}
+728\,C_{2,0}C_{2,8}C_{3,12}-1315\,C_{2,4}^{2}C_{3,12},\\
C_2&=
80\,C_{1,6}^{2}C_{5,8}-80\,C_{1,6}C_{2,8}C_{4,6}
+104\,C_{2,0}C_{2,8}C_{3,12}-125\,C_{2,4}^{2}C_{3,12},\\
C_3&=80\,C_{1,6}^{2}C_{5,8}-40\,C_{1,6}C_{2,4}C_{4,10}
+56\,C_{2,0}C_{2,8}C_{3,12}-55\,C_{2,4}^{2}C_{3,12}. 
\end{aligned}
$$
which provide cusp forms with character 
of weight $(20,32)$ and these are divisible by
$\chi_5^5$ and thus give cusp forms of weight $(20,7)$ 
generating $S_{20,7}(\Gamma_2)$. 
By restriction to the diagonal
one sees that there is just a $1$-dimensional space of forms vanishing on the diagonal.
\end{proof}

The case of weight $(24,2)$ is dealt with in a similar way. 

\begin{proposition}
We have $\dim S_{24,2}(\Gamma_2[2],\epsilon)=2$.
\end{proposition}
\begin{proof}
We know that $\dim S_{24,7}(\Gamma_2)=5$ and we can construct a
basis using the procedure described in Section \ref{covariants}.
In the case at hand we have $A[39,15]$ occurring in 
${\rm Sym}^9({\rm Sym}^6(V))$ with multiplicity~$13$ 
and this gives a subspace of $S_{24,42}(\Gamma_2,\epsilon)$
of dimension~$13$. One checks that there is a $5$-dimensional subspace
of forms vanishing with multiplicity~$7$ along the diagonal and
dividing by $\chi_5^7$ gives a $5$-dimensional subspace of 
$S_{24,7}(\Gamma_2)$, hence the whole space. 
Again one checks that there is a $2$-dimensional
space of forms vanishing on the diagonal and we can divide these forms 
by $\chi_5$.
So the two generators of $S_{24,2}(\Gamma_2,\epsilon)$ 
are defined by the covariants $C_1$ and $C_2$ given respectively by
$$
\begin{aligned}
&-499408\,C_{1,6}^{2}C_{2,0}^{2}C_{3,12}-1505385\,C_{1,6}^{3}C_{6,6}^{(1)}
-14727825\,C_{1,6}^{2}C_{2,4}C_{5,8}+6916455\,C_{1,6}^{2}C_{2,8}C_{5,4}\\
&-5728590\,C_{1,6}^{2}C_{3,12}C_{4,0}+6972210\,C_{1,6}C_{2,0}C_{2,8}C_{4,10}
+4257120\,C_{1,6}C_{2,0}C_{3,6}C_{3,12}\\
&+2182950\,C_{2,8}^{2}C_{5,8}+11708550\,C_{2,8}C_{3,6}C_{4,10}+595350\,C_{2,8}C_{3,12}C_{4,4}
+35171325\,C_{3,6}^{2}C_{3,12}\\
&-400950\,C_{3,8}^{3}\\
\end{aligned}
$$
and
$$
\begin{aligned}
&-42235648\,C_{1,6}^{2}C_{2,0}^{2}C_{3,12}+4434583545\,C_{1,6}^{3}C_{6,6}^{(1)}
+580982220\,C_{1,6}^{3}C_{6,6}^{(2)}\\
&+4919972400\,C_{1,6}^{2}C_{2,4}C_{5,8}+4827362400\,C_{1,6}^{2}C_{3,12}C_{4,0}
-3504891600\,C_{1,6}C_{2,0}C_{2,8}C_{4,10}\\
&+1245336960\,C_{1,6}C_{2,0}C_{3,6}C_{3,12}-4131252720\,C_{2,8}^{2}C_{5,8}-
24904998720\,C_{2,8}C_{3,6}C_{4,10}\\
&-281640240\,C_{2,8}C_{3,12}C_{4,4}-58751907480\,C_{3,6}^{2}C_{3,12}+1375354080\,C_{3,8}^{3}\, .
\end{aligned}
$$
\end{proof}
The order of vanishing of $F_{C_i}$ along $\mathfrak{H}_1\times
\mathfrak{H}_1$ is $8$, so we can divide it by $\chi_{5}^8$ and 
we get two cusp forms, $\chi_{24,2}^{(i)}$, ($i=1,2$) 
of weight $(24,2)$ on $\Gamma_2$ with character.
We set $\chi_{24,2}^{(1)}=-12150\, F_{C_1}/\chi_5^8$ and
$\chi_{24,2}^{(2)}=-675(5368\, F_{C_1} +5 \, F_{C_2})/31528 \chi_5^8$.
Then their Fourier expansions are given by 
$$
\chi_{24,2}^{(1)}=
\left(
\begin{smallmatrix}
0\\
0\\
0\\
104(R-R^{-1})\\
1092(R+R^{-1})\\
3640(R-R^{-1})\\
0\\
-27678(R-R^{-1})\\
-58905(R+R^{-1})\\
-2916(R-R^{-1})\\
148470(R+R^{-1})\\
190778(R-R^{-1})\\
0\\
\vdots
\end{smallmatrix}
\right)Q_1Q_2+\cdots\, ,
\quad 
\chi_{24,2}^{2}(\tau)=
\left(
\begin{smallmatrix}
0\\
0\\
0\\
0\\
0\\
0\\
0\\
2(R-R^{-1})\\
17(R+R^{-1})\\
60(R-R^{-1})\\
110(R+R^{-1})\\
98(R-R^{-1})\\
0\\
\vdots
\end{smallmatrix}
\right)Q_1Q_2+\cdots \, ,
$$
where
$Q_1=e^{i\pi \tau_1}$, $Q_2=e^{i\pi \tau_2}$, $R=e^{i\pi \tau_{12}}$.
The action of $\iota=
(\begin{smallmatrix} a & 0 \\ 0 & d\\ \end{smallmatrix})\in \Gamma_2$ with
$a=d= (\begin{smallmatrix} 0 & 1 \\ 1 & 0\\ \end{smallmatrix})$
implies that the $i$th coordinate is equal to  $(-1)^{k+1}$ times 
the $(j+1-i)$th coordinate, which gives the non-displayed coordinates .
A Hecke eigenbasis of the space
$S_{24,2}(\Gamma_2,\epsilon)$ is:
\begin{align*}
F_1=&\,439 \chi_{24,2}^{(1)} +(114847+650\sqrt{106705})\, \chi_{24,2}^{(2)}\\
F_2=& \, 439\chi_{24,2}^{(1)} +(114847-650\sqrt{106705})\, \chi_{24,2}^{(2)}
\end{align*}
with eigenvalues

\begin{footnotesize}
\smallskip
\vbox{
\bigskip\centerline{\def\quad{\hskip 0.6em\relax}
\def\quod{\hskip 0.5em\relax }
\vbox{\offinterlineskip
\hrule
\halign{&\vrule#&\strut\quod\hfil#\quad\cr
height2pt&\omit&&\omit &&\omit &\cr
& $p$ &&
$\lambda_p(F_1)$  && $\lambda_{p^2}(F_1)$ &\cr
\noalign{\hrule}
height2pt&\omit&&\omit &&\omit &\cr
& 3 && $287880-4800\sqrt{106705}$ && $545747143689-2293459200\sqrt{106705}$ &\cr
& 5 && $711981900+1555200\sqrt{106705}$ && -- &\cr
& 7 && $-41070905840+92534400\sqrt{106705}$  && -- &\cr
& 11 && $ -10344705071976 + 4819953600\sqrt{106705}$  && --  &\cr
} \hrule}
}}
\end{footnotesize}

\noindent
in perfect agreement with the eigenforms being Yoshida lifts. Indeed
a basis of the space $S_{26}(\Gamma_0[2])^{\text{new}}$ is given by
\begin{align*}
f&=q - 4096\, q^2 + 97956\, q^3 + 16777216\, q^4 + 341005350\, q^5
 - 401227776\, q^6 +\cdots\\
g&=q + 4096\, q^2 + (2048-a/2)\, q^3 + 16777216\, q^4 + 
(431848374+162\, a)\, q^5 +\cdots \\
g'&=q + 4096\, q^2 + (2048+a/2)\, q^3 + 16777216\, q^4 + (431848374-162\, a)\, q^5 +\cdots \, ,
\end{align*}
where $ a=-375752+9600\sqrt{106705} $,  
and 
$f,g' \in S^{-}_{26}$ and $g \in S_{26}^{+}$. 
Then we check for example that
\begin{align*}
\lambda_5(F_1)&=711981900+1555200\sqrt{106705}=a_5(f)+a_5(g)\\
\lambda_5(F_2)&=711981900-1555200\sqrt{106705}=a_5(f)+a_5(g').
\end{align*}

\begin{proposition}
One has  
$\dim S_{26,2}(\Gamma_2,\epsilon)=1=\dim S_{28,2}(\Gamma_2,\epsilon)$ 
and  $\dim S_{30,2}(\Gamma_2,\epsilon)=2$.
\end{proposition}
\begin{proof}
The proof of this proposition is similar to the above. For the first statement
we consider the 
space $S_{26,7}(\Gamma_2)$ 
which has dimension $6$ and construct a basis of this space
using covariants associated to $A[43,17]$ in ${\rm Sym}^{10}({\rm Sym}^6(V))$
which occurs with multiplicity $17$, thus giving rise to 
a $17$-dimensional subspace
of $S_{26,47}(\Gamma_2)$. By restricting along the diagonal one 
checks that there is a $6$-dimensional 
subspace of cusp forms divisible by $\chi_5^8$ leading to the construction of
$S_{26,7}(\Gamma_2)$. 
Again by restricting to the diagonal one sees that there is exactly a 
$1$-dimensional subspace of this space that vanish along the diagonal.
By dividing by $\chi_5$ we thus find the space $S_{26,2}(\Gamma_2,\epsilon)$.

For weight $(28,2)$
we  now use the representation  $A[47,19]$ that occurs 
with multiplicity $23$ in ${\rm Sym}^{11}({\rm Sym}^6(V))$ 
and leading to a $23$-dimensional
subspace of $S_{28,52}(\Gamma_2)$ in which we find a $7$-dimensional subspace
of forms divisible by $\chi_5^9$ and division gives forms that
generate  $S_{28,7}(\Gamma_2)$. In this space the subspace of forms divisible
by $\chi_5$ is of dimension $1$, proving our claim. 

In the case of weight $(30,7)$ 
the $9$-dimensional space $S_{30,7}$ is constructed using covariants
resulting from $A[51,21]$ that occurs with multiplicity $31$ 
in ${\rm Sym}^{12}({\rm Sym}^6(V))$ leading to a space of dimension $31$ of
cusp forms of weight $(30,57)$. There is a $9$-dimensional subspace 
of cusp forms divisible by $\chi_5^{10}$ and we thus generate $S_{30,7}(\Gamma_2)$. It turns out that there is a $2$-dimensional subspace of forms divisible
by $\chi_5$ and this proves the claim. 
\end{proof}

For all the cases treated we can check our construction by verifying  
that the Hecke eigenvalues
for $p=3,5,7,11,13, 17$ agree with the forms being Yoshida lifts 
like we indicated for $j=12$ and $j=24$.

In a forthcoming paper (\cite{C-F-vdG2}) we shall use the relation with 
covariants to describe modules of forms with a character.
\end{section}
\begin{section}{Modular Forms of Weight $(j,2)$ on $\Gamma_2$}
In this section we explain how we checked that $S_{j,2}(\Gamma_2)=(0)$
for $j\leq 52$.
We begin with a simple lemma. Recall that we have maps 
$M \to {\mathcal C} {\buildrel \mu \over \longrightarrow} M$.

\begin{lemma}
Let $f \in M_{j,k}(\Gamma_2)$. Then there exists a covariant $c_f$ of degree
$(d,j)$ with $d\leq j/2+k$ such that $f=\nu(c_f)=\mu(c_f)/\chi_5^d$.  
If $f$ is a cusp form then there is a covariant $c_f^{\prime}$ 
of degree $\leq j/2+k-10$
such that $f=\mu(c_f^{\prime})/\chi_5^r$ for some $r$.
\end{lemma}
\begin{proof}
The first statement follows directly from \cite{C-F-vdG}. If $f$ is a cusp
form then the covariant it defines vanishes on the discriminant locus.
But then the covariant $c_f$ is divisible by the discriminant, and $\mu(c_f)$ by $\chi_5^{d+2}$.
\end{proof}

This makes it possible to check the existence of a non-zero form $f \in
S_{j,2}(\Gamma_2)$ by checking whether the forms of weight 
$(j,2+5d)$ provided via $\mu$ by the 
non-zero covariants of degree $(d,j)$ with $d\leq j/2-8$ are 
divisible by $\chi_5^{d}$. We applied this for values of $j\leq 52$ 
using the covariants of degree $d\leq 18$.
For smaller values of $j$ other methods of showing that $S_{j,2}(\Gamma_2)=(0)$
are available.
We sketch some methods below. In this way we checked that
$S_{j,2}(\Gamma_2)=(0)$ for $j\leq 52$.

\medskip

Another method is to construct a basis of $S_{j,7}(\Gamma_2,\epsilon)$ 
by using covariants.  
We then check the divisibility by
$\chi_5$ of elements in $S_{j,7}(\Gamma_2,\epsilon)$ 
by restricting the modular forms in this space to
the diagonal.

As an illustration we give the proof for the case $j=24$.
We construct a basis of the $9$-dimensional space
$S_{24,7}(\Gamma_2,\epsilon)$.
For this we use the
covariants associated to the $A[54,30]$-isotypical component of  ${\rm Sym}^{14}({\rm Sym}^6(V))$. The representation $A[54,30]$ occurs with multiplicity $65$
and leads to a $65$-dimensional subspace of modular forms of weight $(24,72)$ on
$\Gamma_2$. By restricting to $\mathfrak{H}_1 \times \mathfrak{H}_1$ we
can check that there exists a $9$-dimensional subspace of cusp forms
that are divisible by $\chi_5^{13}$.
This leads to a basis of $S_{24,7}(\Gamma_2,\epsilon)$.
We then check by restriction to $\mathfrak{H}_1 \times \mathfrak{H}_1$ again
that there is no non-trivial element in this
space that is divisible by $\chi_5$. This proves the result for $j=24$.

We carried this out for all the cases $j\leq 52$ and thus proved 
Theorem \ref{vanishingSj2level1}.
\medskip

Sometimes there are other and easier ways to eliminate cases. 
For example, by restricting a modular form of weight $(j,2)$ to the diagonal
we get an element $\sum_{i=0}^{j/2} f_i \otimes f_i'
\in \oplus_{i=0}^{j/2} S_{j+2-i}(\Gamma_1)\otimes S_{2+i}(\Gamma_1)$.
If $j<24$, $j\neq 20$  the spaces in question are zero. 
Therefore a form $f\in S_{j,2}(\Gamma_2)$ will vanish on 
$\mathfrak{H}_1 \times \mathfrak{H}_1$. But then
$f/\chi_5$ will be a holomorphic modular form of weight $(j,-3)$ and this has
to be zero. So $f=0$ for $j\leq 18$ and $j=22$.

As yet another example of eliminating cases we give 
a somewhat different argument for $j=26$. 
We write elements of $F\in S_{j,k}(\Gamma_2)$ as vectors
$F=(F_0,\ldots,F_j)^t$ with the $F_i$ holomorphic functions on $\mathcal{H}_2$,
that is, in a  module of rank $27$ over the ring ${\mathcal F}$
of holomorphic functions on $\mathcal{H}_2$.
Take a basis $s_1,\ldots,s_3$ of $S_{26,6}(\Gamma_2)$ and a basis 
$s_{4},\ldots,s_{12}$ of $S_{26,8}(\Gamma_2)$. 
If there exists a non-zero form $f$ of weight $(26,2)$
then the vectors $E_4\, f$ and $E_6\, f$ are linearly dependent and thus
the exterior product $s_1\wedge \cdots \wedge s_{12}$ must vanish.
By calculating bases of $S_{26,6}(\Gamma_2)$ and $S_{26,8}(\Gamma_2)$
one can check that this exterior product does not vanish. 
So $S_{26,2}(\Gamma_2)=(0)$.   

\end{section}

\begin{section}{Other Small Weights}

We begin with an elementary argument that shows that $S_{j,k}(\Gamma_2[2])=(0)$ 
for $j\leq 8$ and $k\leq 2$.

\begin{proposition}\label{vanishsj812}
For $j\leq 8$ and $k\leq 2$ we have $\dim S_{j,k}(\Gamma_2[2])=0$.
\end{proposition}
\begin{proof}
We need to deal with the cases $j$ even and $k=1$ and $k=2$ only since for other values
$S_{j,k}(\Gamma_2[2])$ vanishes.
We restrict to the ten components of the Humbert surface 
$H_1$ in $\Gamma_2[2]\backslash \mathfrak{H}_2$, one component of which is given by the
diagonal $\tau_{12}=0$. The group $\mathfrak{S}_6$ acts transitively on these ten components.
The stabilizer inside $\mathfrak{S}_6$ of a component of $H_1$ is
an extension of $\mathfrak{S}_3\times \mathfrak{S}_3$ by ${\ZZ}/2{\ZZ}$.

By restricting a modular form $f \in S_{j,1}(\Gamma_2[2])$ 
to a component we get an element of 
$$
\oplus_{r=0}^{j} S_{j+1-r}(\Gamma_1[2])\otimes S_{1+r}(\Gamma_1[2]) 
$$
and for $f\in S_{j,2}(\Gamma_2[2])$ we get an element of
$$
\oplus_{r=0}^{j} S_{j+2-r}(\Gamma_1[2])\otimes S_{2+r}(\Gamma_1[2])  \, .
$$
For $j\leq 8$ and $k=1$ and for $j<8$ and $k=2$ these spaces are zero. 
Thus a form $f\in S_{j,2}(\Gamma_2[2])$ restricts to zero
on all irreducible components of $H_1$, hence is divisible by $\chi_{5}$, 
and so $f$ must be zero.
For $j=8$ and $k=2$ the restriction to $H_1$ gives an injective 
$\mathfrak{S}_6$-equivariant map
$$
S_{8,2}(\Gamma_2[2]) \to \oplus_{i=1}^{10} 
\, {\rm Sym}^2  S_6(\Gamma_1[2]) \, ,
$$
where the action on the right is the induced representation from the extension of
$\mathfrak{S}_3 \times \mathfrak{S}_3$ by ${\ZZ}/2{\ZZ}$ to $\mathfrak{S}_6$.
Now $S_6(\Gamma_1[2])$ is 1-dimensional and of type $s[1^3]$ 
and we check that the representation of $\mathfrak{S}_6$ 
on the 10-dimensional space  
$\oplus_{i=1}^{10} \Sym^2 S_6(\Gamma_1[2])$
is of type s[6]+s[4,2].  
Since $S_{8,2}(\Gamma_2)=(0)$, we conclude that only $S_{8,2}(\Gamma_2[2])^{s[4,2]}$ 
can be non-zero.  If $S_{8,2}(\Gamma_2[2])^{s[4,2]}$ is non-zero, 
then  $S_{8,2}(\Gamma_0[2])$ is non-zero (see \cite[Section 9]{C-vdG-G}). 
By restricting we get an element in a similar decomposition as before but with 
$\Gamma_1[2]$ replaced by $\Gamma_0[2]$.  As we know that all theses spaces are zero, 
we can divide by $\chi_5$.
This contradiction concludes our claim. 
\end{proof}

More generally we have
\begin{proposition}
For $j<12$ we have $S_{j,2}(\Gamma_2[2])=(0)$.
\end{proposition}
\begin{proof}
The space $S_{j,2}(\Gamma_2[2])$ is defined over
${\QQ}$. All the cusps of $\Gamma_2[2]$ are defined over ${\QQ}$
and the action of $\mathfrak{S}_6$ is defined over ${\QQ}$. The q-expansion 
principle says that a modular form in $S_{j,k}(\Gamma_2[2])$ with $k\geq 3$ 
is defined over ${\QQ}$ if its Fourier coefficients at all cusps are defined over ${\QQ}$,
see \cite[Cor.\ 1.6.2, 1.12.2]{Katz} and \cite[p.\ 140]{F-C}.
We apply this to $f\chi_{10}$
with  $f \in S_{j,2}(\Gamma_2[2])$ defined
over ${\QQ}$ and we conclude that the Fourier coefficients of $\sigma(f\chi_{10})$
with $\sigma\in \mathfrak{S}_6$ are real, hence also those of $\sigma(f)$ and 
if $f\neq 0$ we find by looking at the `first' non-zero term in a Fourier expansion 
that
$$
\sum_{\sigma \in \mathfrak{S}_6} \sigma(f)^2
$$
is non-zero and because of $\sigma(f^2)=\sigma(f)^2$ also invariant under $\mathfrak{S}_6$.
Thus it defines a non-zero element of $S_{2j,4}(\Gamma_2)$. So 
$S_{j,2}(\Gamma_2[2])$ implies $S_{2j,4}(\Gamma_2)\neq (0)$. But we know
that $S_{2j,4}(\Gamma_2)=(0)$ for $j<12$.
\end{proof}

\begin{remark}
Note that the argument of the proof shows that our conjecture on the vanishing of
$S_{j,2}(\Gamma_2)$ for all $j$ implies the vanishing of $S_{j,1}(\Gamma_2[2])$ for all $j$. 
\end{remark}

\bigskip

In order to put our evidence for the vanishing of $S_{j,2}(\Gamma_2)$ in perspective
we show a small table that gives for each value of $k$ the smallest $j_0$ such
that $\dim S_{j_0,k}(\Gamma_2)\neq 0$. 

\smallskip
\vbox{
\bigskip\centerline{\def\quad{\hskip 0.6em\relax}
\def\quod{\hskip 0.5em\relax }
\vbox{\offinterlineskip
\hrule
\halign{&\vrule#&\strut\quod\hfil#\quad\cr
height2pt&\omit&&\omit &&\omit&&\omit&&\omit&&\omit&&\omit&\cr
& $k$ && $3$ && $4$ && $5$ && $6$ && $7$ && $8$ & \cr
\noalign{\hrule}
height2pt&\omit&&\omit &&\omit&&\omit&&\omit&&\omit&&\omit&\cr
& $j_0$ && $36$ && $24$ && $18$ && $12$ && $12$ && $6$ & \cr
} \hrule}
}}
We can easily construct the generators of the corresponding spaces.
In \cite{C-vdG} we constructed a generator $\chi_{6,3}$ of $S_{6,3}(\Gamma_2,\epsilon)$
and above we gave the generator $\chi_{12,2}$ of $S_{12,2}(\Gamma_2,\epsilon)$. 
The modular forms $\chi_{12,2}^2$, $\chi_{6,3}\chi_{12,2}$, $\chi_{6,3}^2$,
$\chi_5\chi_{12,2}$ and $\chi_5\chi_{6,3}$
give the generators for $k=4,\ldots,8$.
We end by constructing a generator of $S_{36,3}(\Gamma_2)$; the non-vanishing of this space
plays a role in the appendix. We look in
${\rm Sym}^{11}({\rm Sym}^6(V))$ and at $A[51,15]$ occuring with multiplicity $17$ 
there. We find the covariant
$$
297 \, C_{1,6}^2C_{3,8}^3 -8316 \, C_{1,6}C_{3,8}C_{3,12}C_{4,10}
+4116 \, C_{1,6}C_{3,12}^2C_{4,6} -5488 \, C_{2,0}C_{3,12}^3+
9030 \, C_{2,4}C_{3,8}C_{3,12}^2
$$
giving a form $f$ of weight $(36,48)$ that is divisible by $\chi_5^{9}$
and $f/\chi_5^9$ generates $S_{36,3}(\Gamma_2)$.
As a check we note that the Fourier coefficient at $n=[1,1,1]$ is of the form
$$
[0, 0, 0, 0, 0, 0, 0, 32/6089428125, 464/6089428125, \ldots ]
$$
and the coefficient at $n=[5,5,5]$ is 
$$
[0, 0, 0, 0, 0, 0, 0, -8687121398144/81192375, -125963260273088/81192375, 
\ldots ]
$$
giving the eigenvalue for the Hecke operator $T_5$ as their ratio
$ -20360440776900$, in agreement with the value given in \cite{website}.
\end{section}
\begin{appendix}
\newcommand{\ps}{\par \smallskip}
\bigskip

\begin{section}{by Ga\"etan Chenevier}\label{appA}
\noindent Let $j$ and $k$ be integers with $j\geq 0$, and $\Gamma \subset {\rm Sp}_4(\ZZ)$ a congruence subgroup. Recall that $S_{j,k}(\Gamma)$ denotes the space
of cuspidal Siegel modular forms for the subgroup $\Gamma$ with values in the representation
${\rm Sym}^j\, \otimes \,\det^k$ of ${\rm GL}_2(\CC)$. We first consider the full Siegel modular group $\Gamma_2 = {\rm Sp}_4(\ZZ)$ and provide alternative proofs of the following results : \ps

\begin{proposition} \label{k=1} We have ${S}_{j,1}(\Gamma_2)=0$ for any $j$, and ${S}_{j,2}(\Gamma_2)=0$ for $j \leq 38$.
\end{proposition}

The vanishing of ${S}_{j,2}(\Gamma_2)$ for all $j\leq 52$ is also proved in this paper by Cl\'ery and van der Geer (Theorem \ref{vanishingSj2level1}). The vanishing of ${S}_{j,1}(\Gamma_2)$ was at least known to Ibukiyama,
who asserts in \cite[p. 54]{I} that it is a consequence of the vanishing
of all Jacobi forms of weight $1$ for ${\rm SL}_2(\ZZ)$ proven by Skoruppa \cite[Satz 6.1]{Sk}. Here we shall rather use
automorphic representation theoretic methods.
First we need to fix some notations and make some preliminary remarks.
We denote by $r : {\rm Sp}_4(\CC) \rightarrow {\rm GL}_4(\CC)$
the tautological inclusion. \ps \ps

For a real reductive Lie group $H$
we shall denote the infinitesimal character of the
Harish-Chandra module $U$ by ${\rm inf}\, U$. In the case
$H={\rm GL}_n(\RR)$ (resp. $H={\rm PGSp}_4(\RR)$), and following
Harish-Chandra and Langlands, ${\rm inf} \, U$ may be viewed in a
canonical way as a semisimple conjugacy class in the Lie algebra
$\mathfrak{h}=\mathfrak{gl}_n(\mathbb{C})$
(resp. $\mathfrak{h}=\mathfrak{sp}_4(\mathbb{C})$).
In both cases we may and shall identify this conjugacy class with the
multiset of its eigenvalues in the natural representation of $\mathfrak{h}$.
We denote by ${\rm W}_\RR$ the Weil group of $\RR$
(a certain extension of $\ZZ/2\ZZ$ by $\CC^\times$),
and for any integer $w$ we define ${\rm I}_{w}$ as the
$2$-dimensional representation of ${\rm W}_{\mathbb{R}}$
induced from the unitary character $z \mapsto (z/|z|)^{w}$ of
$\CC^\times$. \ps\ps

(a) As we have ${S}_{j,k}(\Gamma_2)=0$ for any odd $j$ or
for $k\leq 0$, we may once and for all assume
$j \equiv 0 \bmod 2$ and $k>0$. As is well-known, for any such $(j,k)$
there is an irreducible unitary Harish-Chandra module for
${\rm PGSp}_4(\RR)$, unique up to isomorphism,
generated by a highest-weight vector of $K$-type
${\rm Sym}^j \,\otimes \det^k$.
This module, that we shall denote by ${\rm U}_{j,k}$,
is a holomorphic discrete series for $k \geq 3$, a limit of
holomorphic discrete series for $k=2$, and non-tempered for $k=1$;
it is {\it non-generic} in all cases.
We have ${\rm inf}\, {\rm U}_{j,k} = 
\{\frac{j+2k-3}{2}, \,\frac{j+1}{2}, \,-\frac{j+1}{2}, \,-\frac{j+2k-3}{2}\}$.
More precisely, if $\varphi : {\rm W}_\RR \rightarrow {\rm Sp}_4(\CC)$
denotes the Langlands parameter of ${\rm U}_{j,k}$,
then we have $r \circ \varphi \simeq {\rm I}_{j+2k-3} \oplus {\rm I}_{j+1}$
for $k>1$, and
$r \circ \varphi \simeq {\rm I}_{j} \otimes |.|^{1/2} 
\oplus {\rm I}_{j} \otimes |.|^{-1/2}$ for $k=1$ (see e.g. \cite{S}
for a survey of those properties, and the references therein).\ps\ps

The relevance of ${\rm U}_{j,k}$ here is that if $\pi$ is a
cuspidal automorphic representation of ${\rm PGSp}_4$ over
$\mathbb{Q}$ generated by an element of ${S}_{j,k}(\Gamma_2)$,
then the Archimedean component $\pi_\infty$ of $\pi$ is isomorphic
to ${\rm U}_{j,k}$. The other important property of $\pi$ is that
$\pi_p$ is unramified for each prime $p$ (i.e. admits non-zero
invariants under ${\rm PGSp}_4(\ZZ_p)$). As ${\rm PGSp}_4$ is
isomorphic to the split classical group ${\rm SO}_5$ over $\mathbb{Z}$,
we may apply Arthur's theory \cite{arthur} to such a $\pi$. \ps\ps

(b)  One of the main results of Arthur \cite[Thm. 1.5.2]{arthur}
associates to any discrete automorphic representation $\pi$ of
${\rm PGSp}_4$ over $\QQ$ a unique isobaric automorphic
representation $\pi^{\rm GL}$ of ${\rm GL}_4$ over $\mathbb{Q}$,
characterized by the following property :
for any prime $p$ such that $\pi_p$ is unramified,
then $(\pi^{\rm GL})_p$ is unramified as well and its
Satake parameter is the image of the one of $\pi_p$ under the map $r$.
The infinitesimal character of $(\pi^{\rm GL})_\infty$ is the image of
${\rm inf}\, \pi_\infty$ under the derivative of $r$,
namely $\mathfrak{sp}_4(\CC) \rightarrow \mathfrak{gl}_4(\CC)$.
Moreover, there is a unique collection of distinct triplets
$(d_i,n_i,\pi_i)_{i \in I}$, with integers $d_i,n_i \geq 1$ and cuspidal selfdual automorphic representations $\pi_i$ of ${\rm GL}_{n_i}$ with
$$\pi^{\rm GL} \,\simeq \,\boxplus_{i \in I} \,(\,\boxplus_{l=0}^{d_i-1}\,\, \,\pi_i \,\otimes |.|^{\frac{d_i-1}{2}-l}\,)\, \, \, \, \, {\rm and} \,\,\,\,\,4 = \sum_{i \in I} n_i d_i.$$
The selfdual representation $\pi_i$ is symplectic in Arthur's sense if,
and only if, $d_i$ is odd.
All of this is included in \cite[Thm. 1.5.2]{arthur}.\ps\ps

(c) For $\pi$ as in (b), then ${\rm inf}\, \pi_\infty$ is the union,
over all $i \in I$ and all $0 \leq l< d_i$, of the multisets
$\frac{d_i-1}{2}-l\,+\,{\rm inf}\, (\pi_i)_\infty$.
In particular, if we have
$\lambda \in \frac{1}{2}\ZZ$ and $\lambda-\mu \in \ZZ$
for all $\lambda,\mu \in \,{\rm inf}\, \pi_\infty$,
then ${\rm inf}\, (\pi_i)_\infty$ has the same property for each $i$ :
such a $\pi_i$ is called {\it algebraic}.
If $\omega$ is a cuspidal selfdual algebraic automorphic representation of
${\rm GL}_m$ over $\QQ$, then $\omega_\infty$ is tempered
by the Jaquet-Shalika estimates,
and its Langlands parameter is trivial on the central subgroup
$\RR_{>0}$ of ${\rm W}_\RR$
(this is the so-called {\it Clozel purity lemma}, see e.g.  \cite[Chap. VIII Prop. 2.13]{CL}). \ps\ps

(d) The only selfdual cuspidal automorphic representation
$\pi$ of ${\rm GL}_1$ over $\QQ$ such that $\pi_p$ is unramified
for each prime $p$ is the trivial Hecke character $1$
(which is of course selfdual orthogonal).
Moreover, for any integer $k\geq 1$,
the number of cuspidal automorphic representations $\pi$ of
${\rm GL}_2$ over $\QQ$ such that $\pi_p$ is unramified for
each prime $p$, and with
${\rm inf}\, \pi_\infty =\{-\frac{k-1}{2},\frac{k-1}{2}\}$,
is the dimension of the space ${S}_k(\Gamma_1)$ of
cuspidal modular forms of weight $k$ for $\Gamma_1={\rm SL}_2(\ZZ)$.
Indeed, this is well-known for $k>1$, and for $k=1$
it follows from the fact that there is no Maass form
of eigenvalue $1/4$ for ${\rm SL}_2(\ZZ)$
(a fact due to Selberg, see also \cite[Chap. IX \S 3.19]{CL} for a short proof). \ps\ps

\ps\ps

\begin{proof} (of the vanishing of ${S}_{j,1}(\Gamma_2)$ for any $j$).
It is enough to show that there is no discrete automorphic
representation $\pi$ of ${\rm PGSp}_4$ over $\mathbb{Q}$
which is unramified at every prime and with $\pi_\infty \simeq {\rm U}_{j,1}$.
For that we study $\pi^{\rm GL}$,
and the associated collection $(d_i,n_i,\pi_i)_{i \in I}$ given by (b) above.
By (a), the infinitesimal character of $(\pi^{\rm GL})_\infty$ is
{\small $${\rm inf}\, {\rm U}_{j,1} \, =\,\{\,\frac{j+1}{2}, \,\frac{j-1}{2}, \,-\frac{j-1}{2}, \,-\frac{j+1}{2}\}.$$}
\noindent If we have $d_i=1$ for each $i$, then $(\pi^{\rm GL})_\infty$ is
tempered by (c), hence so is $\pi_\infty$ by Arthur's local-global
compatibility \cite[Thm. 1.5.1 (b) \& Thm. 1.5.2]{arthur}, a
contradiction as ${\rm U}_{j,1}$ is non-tempered by (a).
Fix $i \in I$ with $d_i\geq 2$.
If we have $n_i \geq 2$ then we must have $I=\{i\}$ and
$n_i=d_i=2$ by the equality $4 = \sum_i n_i d_i$. By (c)
and the shape of ${\rm inf}\, {\rm U}_{j,1}$ above, we
necessarily have ${\rm inf} \, (\pi_i)_\infty = \{j/2,-j/2\}$.
But this is absurd by (d) and the vanishing
${S}_{j+1}(\Gamma_1)=0$ for any even integer $j \geq 0$.
We have thus $(d_i,n_i,\pi_i)=(2,1,1)$.
Choose $i' \in I-\{i\}$.
As we have $(d_{i'},n_{i'},\pi_{i'}) \neq (d_i,n_i,\pi_i)$,
the previous argument shows $d_{i'}=1$, so $
\pi_{i'}$ is symplectic by (b), which forces $n_{i'}$ to be even, hence the only
 possibility is $n_{i'}=2$ and $I=\{i,i'\}$.
But then the shape of ${\rm inf}\, {\rm U}_{j,1}$ and (c) show
that we have either $j=0$ and ${\rm inf} (\pi_{i'})_\infty = \{ 1/2, -1/2 \}$
or $j=1$ and ${\rm inf} (\pi_{i'})_\infty = \{ 3/2, -3/2 \}$.
Both cases are absurd by (d) as we have ${S}_4(\Gamma_1)={S}_2(\Gamma_1)=0$, and we are done.
\end{proof}

The second assertion of the proposition will be a consequence of the
following two lemmas.

\begin{lemma}\label{Pi}
Let $j \geq 0$ be an even integer.
The integer $\dim {S}_{j,2}(\Gamma_2)$ is the number
of cuspidal, selfdual symplectic, automorphic representations
$\Pi$ of ${\rm GL_4}$ over $\mathbb{Q}$
whose local components $\Pi_p$ are unramified
for each prime $p$, and with ${\rm inf} \, \Pi_\infty\, = \,\{\,\frac{j+1}{2}, \,\frac{j+1}{2}, \,-\frac{j+1}{2}, \,-\frac{j+1}{2}\}$.
\end{lemma}

\begin{proof}
Let $\pi$ be a cuspidal automorphic representation of
${\rm PGSp}_4$ over $\mathbb{Q}$ generated by an arbitrary
Hecke eigenform $F$ in ${S}_{j,2}(\Gamma_2)$.
Consider its associated automorphic representation
$\pi^{\rm GL}$ of ${\rm GL}_4$ and collection of
$(d_i,n_i,\pi_i)$'s as in (b).
We claim that $\pi^{\rm GL}$ is necessarily cuspidal, i.e.
$I=\{i\}$ is a singleton and $d_i=1$,
so that $\Pi=\pi^{\rm GL}$ satisfies all the assumptions of
the statement by (a) and (b). \ps

Let us show first that we have $d_i=1$ for each $i \in I$.
Otherwise, (c) shows that two elements of
${\rm inf}\, {\rm U}_{j,2}$ must differ by $1$, which only happens for $j=0$.
But for $j=0$ an argument similar to the one in the previous proof
shows that if $d_i>1$ then we have $I=\{i,i'\}$ with
$(d_i,n_i,\pi_i)=(2,1,1)$, $n'_i=2$ and
${\rm inf}\, (\pi_i)_{\infty} = \{1/2,-1/2\}$,
which is absurd by the vanishing ${S}_2(\Gamma_1)=0$
and (d), and we are done. As a consequence,
the Langlands parameter of
$(\pi^{\rm GL})_\infty$ is ${\rm I}_{j+1} \oplus {\rm I}_{j+1}$ by (c). \ps

Let us denote by $\psi$ Arthur's substitute for the
{\it global parameter} of the representation $\pi$ of
${\rm PGSp}_{4}$ defined in \cite[Chap. 1 \S 1.4]{arthur}.
We have just proved that $\psi_\infty$ is the tempered
Langlands parameter of ${\rm PGSp}_4(\RR)$ with
$r \,\circ \,\psi_\infty \simeq {\rm I}_{j+1} \oplus {\rm I}_{j+1}$,
{\it i.e.} the Langlands parameter of ${\rm U}_{j,2}$ by (a).
We have $\mathcal{S}_{\psi_\infty} = \ZZ/2\ZZ$ :
the corresponding $L$-packet of ${\rm PGSp}_4(\RR)$
(limit of discrete series) has two elements, namely ${\rm U}_{j,2}$
and a generic limit of discrete series with same infinitesimal character.
We now apply Arthur's multiplicity formula to the element $\pi$
of the global packet $\Pi_\psi$ defined by Arthur.
As we have $d_i=1$ for all $i$, either $\pi^{\rm GL}$
is cuspidal or we have $I=\{i,i'\}$
with $n_i=n_{i'}=2$ and ${\rm inf}\, (\pi_i)_\infty = 
{\rm inf}\, (\pi_{i'})_\infty = \{\frac{j+1}{2},-\frac{j+1}{2}\}$.
If $\pi^{\rm GL}$ is not cuspidal, then according to Arthur's
definitions the natural map
$\mathcal{S}_\psi \rightarrow \mathcal{S}_{\psi_\infty}$ is an
isomorphism of groups of order $2$. But then his multiplicity formula shows that $\pi_\infty$ has to be generic since $\pi_p$ is unramified for each prime $p$, a contradiction as ${\rm U}_{j,2}$ is not generic. (We have shown that $\pi$ is not of ``Yoshida type''.) 
We have thus proved that $\pi^{\rm GL}$ is cuspidal. 
Note that in this case we have $\mathcal{S}_\psi=1$, 
thus by the multiplicity formula again, 
the multiplicity of $\pi$ in the automorphic discrete spectrum of ${\rm PGSp}_4$ is $1$; 
in particular, the Hecke eigenspace of the eigenform $F$ 
we started from has dimension $1$. 
It thus only remains to show that any $\Pi$ 
as in the statement is in the image of the construction of the first paragraph above. \ps

Let $\Pi$ be as in the statement. The Langlands parameter of
$\Pi_\infty$ is the image under $r$ of the one of
${\rm U}_{j,2}$ by (a) and (c).
A trivial application of Arthur's multiplicity formula shows
the existence of a discrete automorphic $\pi$ for ${\rm PGSp}_4$
with $\pi_\infty \simeq {\rm U}_{j,2}$, which is unramified at every prime,
and satisfying $\pi^{\rm GL} \simeq \Pi$. As ${\rm U}_{j,2}$ is tempered,
a classical result of Wallach ensures that $\pi$ is actually cuspidal,
hence generated by an element of ${S}_{j,2}(\Gamma_2)$ :
this concludes the proof.
\end{proof}

\begin{lemma} For any even integer $0 \leq j \leq 38$ there is no $\Pi$ as in Lemma \ref{Pi}. \end{lemma}

In order to contradict the existence of a $\Pi$ as in Lemma~\ref{Pi} for small $j$, and following work of Odlyzko, Mestre, Fermigier, Miller and Chenevier-Lannes, we shall apply the so-called {\it explicit formula} ``\`a la Riemann-Weil'' to a suitable test function $F$ and to the complete Rankin-Selberg $L$-function ${\rm L}(s,\Pi \times \Pi')$, first to $\Pi' = \Pi^{\vee}$ (the contragredient of $\Pi$) and then to some other well-chosen cuspidal automorphic representations $\Pi'$. Let us stress that the analytic properties of those Rankin-Selberg $L$-functions (meromorphic continuation to $\mathbb{C}$, functional equation, determination of the poles, and boundedness in vertical strips away from the poles) which have been established by Gelbart, Jacquet, Shalika and Shahidi, will play a crucial role in the argument. \ps

\begin{proof}  It will be convenient to follow the exposition of
the explicit formula given in \cite[Chap. IX \S 3]{CL},
which is designed for this kind of applications, and from which
we shall borrow our notations. In particular, we choose for the
test function $F$ the scaling of Odlyzko's function which is denoted by
${\rm F}_{\lambda}$ in \cite[Chap. IX \S 3.16]{CL}, and we denote by
${\rm K}_{\infty}$ the Grothendieck ring of finite dimensional
complex representations of the quotient of the compact group
${\rm W}_{\mathbb{R}}$ by its central subgroup $\RR_{>0}$, and by
${\rm J}_{F} : {\rm K}_{\infty} \rightarrow \mathbb{R}$
the concrete linear form associated to $F$ defined in
Proposition-Def. 3.7 {\it loc. cit}. \ps

Let $j\geq 0$ be an even integer and let $\Pi$ be as in Lemma~\ref{Pi}.
As already explained, it follows from (c) that the Langlands parameter
of $\Pi_\infty$ is ${\rm I}_{j+1} \oplus {\rm I}_{j+1}$,
whose square in the ring
${\rm K}_\infty$ is  $4\, ( {\rm I}_{2j+2}+ {\rm I}_{0})$.
The explicit formula applied to $\Pi \times \Pi^\vee$
leads to the inequality \cite[Chap. IX Cor. 3.11 (i)]{CL} :
$$ {\rm J}_{{\rm F}_{\lambda}}({\rm I}_{2j+2})+ 
{\rm J}_{{\rm F}_{\lambda}}({\rm I}_{0}) \leq \frac{2}{\pi^{2}} \lambda
$$
for all $\lambda >0$. As ${\rm J}_{{\rm F}_{\lambda}}({\rm I}_{w})$
is a non-increasing function of $w$,
the truth of the proposition for $j \leq 34$ is a consequence of the following numerical computation for $\lambda = 3.3$$${\rm J}_{{\rm F}_{\lambda}}({\rm I}_{70})+ {\rm J}_{{\rm F}_{\lambda}}({\rm I}
_{0}) \simeq 0.679 \, \, \, \, \, \, {\rm and}\, \, \, \, \,  \frac{2}{\pi^{2}} 
\lambda \simeq 0.669$$(the values are given up to
$10^{-3}$, and the left-hand side has been computed
using the formula for ${\rm J}_{F_{\lambda}}$ given in \cite[Chap. IX Prop. 3.17]{CL}).
\ps

We now explain how to deal with the cases $j=36$ and $38$.
By Tsushima's formula (proved for $k=3$
independently by Petersen and Ta\"ibi),
we know that the first value of $j$ such that
${S}_{j,3}(\Gamma_2)$ is non-zero is $j=36$,
in which case it has dimension $1$.
Let $\pi$ be the cuspidal automorphic representation of ${\rm PGSp}_4$ over
$\QQ$ generated by ${S}_{36,3}(\Gamma_2)$ and set $\Pi' = \pi^{\rm GL}$.
This $\Pi'$ is a selfdual cuspidal representation by
\cite[Chap. IX Prop. 1.4]{CL}, and the Langlands parameter
of $\Pi'_\infty$ is ${\rm I}_{39} \oplus {\rm I}_{37}$ (the existence of such a $\Pi'$ actually ``explains'' why the argument above breaks down at $j=36$, as the Langlands parameter of $\Pi'_\infty$ is ``close'' to ${\rm I}_{37}\oplus {\rm I}_{37}$). We now apply the explicit formula to $\Pi \times \Pi^\vee$, $\Pi' \times {\Pi'}^\vee$ and $\Pi \times {\Pi'}^\vee$. It leads to a simple criterion, given in \cite[Chap. IX Scholie 3.26]{CL}, for $\Pi$ not to exist : the explicit quantity denoted there by ${\rm t}(V,V',\lambda)$ has to be $\geq 0$ for all $\lambda$, where $V$ and $V'$ are the respective Langlands parameter of $\Pi_\infty$ and $\Pi'_\infty$. But a computation gives
$${\rm t}({\rm I}_{37} + {\rm I}_{37},{\rm I}_{39} + {\rm I}_{37},4) \simeq -0.429\,\,\,\,\, {\rm and} \,\,\, \, \, {\rm t}({\rm I}_{39}+ {\rm I}_{39},{\rm I}_{39} + {\rm I}_{37},4) \simeq -0.039$$
(these values are given up to $10^{-3}$) which are both $<0$.
This concludes the proof.
\end{proof}

\begin{remark} {\rm (i)} As we have ${S}_{j,3}(\Gamma_2)=0$ for $j<36$ by Tsushima's formula, the vanishing of ${S}_{j,2}(\Gamma_2)$ for $j \leq 38$ is a very mild evidence toward Conjecture 1.1 of Cl\'ery and van der Geer. \ps
{\rm (ii)}  Lemma \ref{Pi} and the explicit formula can also be used to obtain upper bounds on $\dim {S}_{j,2}(\Gamma_2)$. Indeed, keeping the notations in the above proof, and applying \cite[Chap. IX Cor. 3.14]{CL} (due to Ta\"ibi), we get that the inequality $$({\rm J}_{{\rm F}_{\lambda}}({\rm I}_{2j+2})+ {\rm J}_{{\rm F}_{\lambda}}({\rm I}_{0})) \dim {S}_{j,2}(\Gamma_2) \leq \frac{2}{\pi^{2}} \lambda$$ holds for all $\lambda >0$. When the parenthesis on left hand side is $> 0$, which happens (for big enough $\lambda$) for all $j \leq 138$, we obtain an explicit upper bound for $\dim {S}_{j,2}(\Gamma_2)$. For instance, we get $\dim {S}_{j,2}(\Gamma_2) \leq 1$ for all $j < 54$ and  $\dim {S}_{j,2}(\Gamma_2) \leq 2$ for all $j < 66$ (choose respectively $\lambda=5$ and $\lambda=6$).
\end{remark}
\ps

\noindent Our last and main result concerns the kernel $\Gamma_2[2]$ of the reduction ${\rm Sp}_4(\ZZ) \rightarrow {\rm Sp}_4(\ZZ/2\ZZ)$.

\begin{theorem}\label{mainthmapp} We have ${S}_{j,1}(\Gamma_2[2])=0$ for any $j$.
\end{theorem}

\noindent Our proof will be an elaboration of the one of Proposition \ref{k=1}. We shall also use the vanishing ${S}_{j,1}(\Gamma_2[2])=0$ for $j\leq 8$, proved by Cl\'ery and van der Geer in this paper (Proposition \ref{vanishsj812}).

\begin{proof}  As we have $-1 \in \Gamma_2[2]$ we may also assume $j$ is even. Let us denote by $J$ the principal congruence subgroup of ${\rm PGSp}_4(\mathbb{Z}_2)$ and by $\mathbb{A}$ the adele ring of $\mathbb{Q}$; we easily check ${\rm PGSp}_4(\mathbb{A}) = {\rm PGSp}_4(\mathbb{Q}) \cdot ({\rm PGSp}_4(\mathbb{R})^0 \times J \times \prod_{p\neq 2} {\rm PGSp}_4(\mathbb{Z}_p)).$
Moreover, classical arguments show that we have $S_{j,1}(\Gamma_2[2])=0$ if, and only if, there is no cuspidal automorphic representation $\pi$ of ${\rm PGSp}_4$ over $\mathbb{Q}$ which is unramified at every odd prime, such that $\pi_2$ has a non-zero invariant under $J$, and with $\pi_\infty \simeq {\rm U}_{j,1}$. Thus we fix such a $\pi$ and consider $\pi^{\rm GL}$, as well as the associated collection $(d_i,n_i,\pi_i)_{i \in I}$, given by (b) above. \ps
By the same argument as in the case $\Gamma=\Gamma_2$, there exists $i \in I$ with $d_i \geq 2$. If we have $n_i=1$, which forces ${\rm inf}\, \pi_i = \{0\}$, we must have $d_i=2$ and $j \in \{0,2\}$ by the shape of ${\rm inf}\, {\rm U}_{j,1}$. But this is a contradiction as ${S}_{j,1}(\Gamma_2[2])=0$ for $j=0,2$. So we must have $n_i=d_i=2$, $I=\{i\}$, and $\pi_i$ is orthogonal with ${\rm inf} (\pi_i)_\infty = \{-j/2,j/2\}$.
\ps
The classification of orthogonal cuspidal automorphic representations of ${\rm GL}_2$ over $\QQ$, a very special case of Arthur's results, is well-known. First of all, the central character of such a representation has order $2$, hence corresponds to some uniquely defined quadratic extension $K$ of $\QQ$. Moreover, for any Hecke character $\chi$ of $K$ which is trivial on the idele group of $\QQ$, and with $\chi^2 \neq 1$, the automorphic induction of $\chi$ to $\QQ$ is an orthogonal cuspidal automorphic representation of ${\rm GL}_2$ over $\QQ$ that we shall denote by ${\rm ind}(\chi)$. It turns out that they all have this form, and that we have moreover ${\rm ind}(\chi) \simeq {\rm ind}(\chi')$ if, and only if, we have $\chi=\chi'$ or $\chi^{-1} = \chi'$. Last but not least, to any $\chi$ as above Arthur associates a global packet $\Pi(\chi)=\bigotimes'_v \Pi_v(\chi_v)$ of irreducible admissible representations of ${\rm PGSp}_4$ over $\QQ$, whose elements are exactly the discrete automorphic representations $\omega$ which satisfy $\omega^{\rm GL} = {\rm ind}(\chi) \otimes |.|^{1/2} \boxplus {\rm ind}(\chi) \otimes |.|^{-1/2}$ ({\it Soudry type}); in this ``stable case'' any element of $\Pi(\chi)$ is automorphic by Arthur's multiplicity formula.
\ps\ps

Going back to our specific situation, let $K$ and $\chi$ be such that $\pi_i \simeq {\rm ind}(\chi)$. Since $\pi_i$ is unramified outside $2$, then so is $K$ and we necessarily have
$$K=\QQ(\sqrt{d})\, \, \, {\rm with}\, \, \, d\,\in \{-2,\,-1,\,2\}.$$
We first claim $d \neq 2$. Indeed, a Hecke character of real quadratic field has the form $|.|^{s_0} \chi_0$ with $\chi_0$ a finite order character and $s_0 \in \CC$. We would thus have $\{s_0,s_0\} = {\rm inf} (\pi_i)_\infty  =\{j/2,-j/2\}$, which implies $s_0=j=0$, which is again absurd. So $K$ is imaginary quadratic. As $\chi_\infty$ is trivial on $\RR^\times$ by assumption on $\chi$, and up to replacing $\chi$ by $\chi^{-1}$ if necessary, the shape of ${\rm inf}\, (\pi_i)_\infty$ implies then $\chi_\infty(z) = (z/\overline{z})^{j/2}$. \ps

Here comes the main trick. Let $\eta$ be the Hecke character of the statement of Lemma \ref{consthecke} below, and set $w_K=|\mathcal{O}_K^\times|$. We may assume $j\geq 2$, so there are unique integers $r$ and $j'/2$, with $r\geq 0$ and $1 \leq j'/2 \leq w_K$, such that $j/2 \,=\,r \,w_K \, +\, j'/2$.  Consider the Hecke character $\chi' = \chi \eta^{-r}$ of $K$. It is obviously trivial on the idele group of $\QQ$, and it satisfies $(\chi')^2 \neq 1$ as we have $\chi'_\infty(z)= (z/\overline{z})^{j'/2}$ with $j'>0$. Consider now the packet $\Pi(\chi')$. As we have $\chi_2=\chi'_2$, the local Arthur packets $\Pi_2(\chi_2)$ and $\Pi_2(\chi'_2)$ do coincide. As $\pi$ belongs to $\Pi(\chi)$, its local component $\pi_2$ also belongs to $\Pi_2(\chi_2)$. We may thus consider a representation $\pi'=\bigotimes'_v \pi'_v$ in $\Pi(\chi')$ with $\pi'_2\simeq \pi_2$, with $\pi'_p$ unramified for all odd prime $p$ (since $\chi'_p$ is unramified for such a $p$), and with $\pi'_\infty \simeq {\rm U}_{j',1}$. This last property holds because the Langlands packet associated to the Arthur packet $\Pi_\infty(\chi'_\infty)$, which is included in $\Pi_\infty(\chi'_\infty)$ by \cite[Prop. 7.4.1]{arthur}, is the one of ${\rm U}_{j',1}$ by Remark (a) above. As already explained, the representation $\pi'$ is discrete automorphic by Arthur, and even cuspidal as its Archimedean component is tempered. As $\pi'_2 \simeq \pi_2$ has non-zero invariants under the principal congruence subgroup $J$ of ${\rm PGSp}_4(\mathbb{Z}_2)$, it follows that $\pi'$ is generated by an element in ${S}_{j',1}(\Gamma_2[2])$ by the first paragraph above. But now we have the inequality
$j' \,\leq \,2 \,w_K \,\leq \,8$, a contradiction by the vanishing ${S}_{j',1}(\Gamma_2[2])=0$ for $j'\leq 8$. \end{proof}

\begin{lemma}\label{consthecke} Let $K=\QQ(\sqrt{d}) \subset \CC$ with $d=-1,-2$ and set $w_K=|\mathcal{O}_K^\times|$. There is a Hecke character $\eta$ of $K$ which is unramified outside $\{2,\infty\}$ and which satisfies $\eta_2=1$ and $\eta_\infty(z)=(z/\overline{z})^{w_K}$ for all $z \in K_\infty^\times$. Moreover, $\eta$ is trivial on the idele group of $\QQ$. \end{lemma}

\begin{proof} Denote by $\mathbb{A}_F$ the adele ring of the number field $F$. As $\mathcal{O}_K$ has class number $1$ we have the decomposition $\mathbb{A}_K^\times = K^\times \cdot (K_\infty^\times \times K_2^\times \times \prod_{v\neq 2,\infty} {\mathcal{O}^\times_{K_v}})$. This implies first the (unrequired) uniqueness of $\eta$, and shows that its existence is equivalent to the fact that the morphism $z \mapsto (z/\overline{z})^{w_K}, \CC^\times \rightarrow \CC^\times$, is trivial on the subgroup $(\mathcal{O}_K[1/2])^\times$. This is indeed the case as this latter group is generated by $\mathcal{O}_K^\times$ and by some element $\pi \in \mathcal{O}_K$ with norm $2$ and which satisfies $\pi/\overline{\pi} \in \mathcal{O}_K^\times$. The last assertion follows from the equality $\mathbb{A}_\QQ^\times = \QQ^\times \cdot (\RR^\times \times \prod_p \ZZ_p^\times)$ and the properties of $\eta$.
\end{proof}

\end{section}
\begin{section}{Proof of Theorem \ref{YoshidaLifts} by Ga\"etan Chenevier}

In this second appendix, we explain how to deduce Theorem \ref{YoshidaLifts} stated in the paper of Cl\'ery and van der Geer from the works of R\"osner \cite{Roesner} and Weissauer \cite{weissauerbook}, for the convenience of the reader. \par \smallskip

We first recall some results on Yoshida lifts taken from Weissauer's work \cite[Chap. 4 \& 5]{weissauerbook}. Let us fix $f$ and $g$ two non-proportional elliptic eigen newforms of same even weight $j+2$, and let us denote by $\pi$ and $\pi'$ the (distinct) cuspidal automorphic representations of ${\rm GL_2}(\mathbb{A})$ that they generate, $\mathbb{A}$ being the adele ring of $\QQ$. In particular, $\pi_\infty$ and $\pi'_\infty$ are isomorphic discrete series, and we may assume that $\pi$ and $\pi'$ are normalized such that this discrete series has trivial central character (as $j$ is even). For Yoshida lifts $Y(f,g)$ of $f$ and $g$ to exist, we also need to assume that $f$ and $g$ have the same {\it nebentypus}, i.e. that $\pi$ and $\pi'$ have the same central character. \par \smallskip

Let $\Pi(\pi,\pi')=\bigotimes'_v \Pi_v(\pi_v,\pi'_v)$ be the restricted tensor product, over all the places $v$ of $\QQ$, of the local $L$-packet $\Pi_v(\pi_v,\pi'_v)$ of irreducible admissible representations of ${\rm GSp}_4(\QQ_v)$ associated to the pair $\{\pi_v,\pi'_v\}$ by Weissauer. For each place $v$ of $\QQ$, this local $L$-packet has either $1$ or $2$ elements, including a unique {\it generic} element; it has another element if, and only if, both $\pi_v$ and $\pi'_v$ are discrete series \cite[\S 4.10.3]{weissauerbook} \cite[Lemma 4.5]{Roesner}. In particular, the (Langlands) Archimedean packet $\Pi_\infty(\pi_\infty,\pi'_\infty)$ has two elements, the non-generic one being the holomorphic limit of discrete series ${\rm U}_{j,2}$ recalled in appendix \ref{appA}. Also, if both $\pi$ and $\pi'$ are unramified at the finite place $v$ then $\Pi_v(\pi_v,\pi'_v)$ is a singleton (thus $\Pi(\pi,\pi')$ is finite). The multiplicity formula proved by Weissauer \cite[Thm. 5.2, p. 186]{weissauerbook} states that an element $\varpi$ of $\Pi(\pi,\pi')$ is discrete automorphic if, and only if, there is an even number of places $v$ such that $\varpi_v$ is non-generic. He also shows that such an element has multiplicity one in the discrete spectrum of ${\rm GSp}_4$; it is necessarily cuspidal as the two elements of $\Pi_\infty(\pi_\infty,\pi'_\infty)$ are tempered. \par \smallskip
Let us denote by $J \subset {\rm GSp}_4(\ZZ_2)$ the principal congruence subgroup. Some classical arguments show that the Yoshida lifts $Y(f,g)$ which belong to the space $YS_{j,2}^{s[w]}$ of the statement are in natural bijection with certain vectors of the finite part of the cuspidal automorphic representations $\varpi$ in $\Pi(\pi,\pi')$ having the following properties :\par \smallskip
(i) $\varpi_\infty \simeq {\rm U}_{j,2}$, \par \smallskip
(ii) $\varpi_p^{{\rm GSp}_4(\ZZ_p)} \neq 0$ (in which case we have $\, \dim \, \varpi_p^{{\rm GSp}_4(\ZZ_p)}=1$), \par \smallskip
(iii) the $s[w]$-isotypic component $\varpi_2^{J}$ is non-zero.\par \medskip
\noindent
More precisely, the Yoshida lifts $Y(f,g)$ corresponding to such a $\varpi$ form a linear subspace $YS_{j,2}^{s[w]}[\varpi] \subset YS_{j,2}^{s[w]}$ isomorphic to the $s[w]$-isotypic component of $\varpi_2^{J}$ as an $S_6$-representation. The space $YS_{j,2}^{s[w]}$ of the statement is then the direct sum of its subspaces $YS_{j,2}^{s[w]}[\varpi]$ where $f,g$ and $\varpi$ vary, with $\varpi$ cuspidal automorphic satisfying (i), (ii) and (iii). \par \smallskip

We still fix elliptic newforms $f$ and $g$ as above, hence $\pi$ and $\pi'$ as well. By \cite[Cor. 4.14]{Roesner}, we know first that $\Pi_p(\pi_p,\pi'_p)$ has an element with non-zero invariants under ${\rm GSp}_4(\mathbb{Z}_p)$ if, and only if, both $\pi_p$ and $\pi'_p$ are unramified. Moreover, the same corollary asserts that if $\Pi_2(\pi_2,\pi'_2)$ has an element with non-zero $J$-invariants, then both $\pi_2$ and $\pi'_2$ have non-zero invariants under the principal congruence subgroup of ${\rm GL}_2(\ZZ_2)$. As a first consequence, if $\varpi \in \Pi(\pi,\pi')$ does satisfy (ii) and (iii) then $f$ and $g$ are newforms on $\Gamma_0(N)$ with $N|4$ (and both $\pi$ and $\pi'$ have a trivial central character). Moreover, by the statement recalled above concerning the multiplicity formula ("even parity of the number of non-generic places"), there is at most one cuspidal automorphic $\varpi \in \Pi(\pi,\pi')$ satisfying (i), (ii) and (iii) above, and it has the property that $\varpi_2$ is the non-generic element of $\Pi_2(\pi_2,\pi'_2)$. In particular, this latter $L$-packet has two elements and both $\pi_2$ and $\pi'_2$ are discrete series of ${\rm GL}_2(\QQ_2)$ (thus neither $f$ nor $g$ can have level $1$). \par \smallskip

By R\"osner \cite[Lemma 5.22]{Roesner}, there are only three possibilities for the isomorphism class of a representation of ${\rm PGL}_2(\mathbb{Q}_2)$ generated by an elliptic newform of level $\Gamma_0(2)$ or $\Gamma_0(4)$, namely the {\it Steinberg} representation {\rm St} and its unramified quadratic twist ${\rm St}'$ in level $\Gamma_0(2)$, and a certain supercuspidal representation ${\rm Sc}$ in level $\Gamma_0(4)$. In particular, there are all discrete series. Assume now that $\sigma$ and $\sigma'$ are (possibly equal) elements of the set $\{{\rm St}, {\rm St'}, {\rm Sc}\}$ and let $\tau$ be the non-generic element of $\Pi_2(\sigma,\sigma')$. View the finite dimensional vector space $\tau^J$ as a representation of $S_6$. In order to prove the theorem it only remains to show that either $\tau^J$ is $0$ or we are in exactly one of the following situations : \par \smallskip

(a) $\{\sigma,\sigma'\}=\{{\rm St},{\rm St'}\}$ and $\tau^J \simeq s[1^6]$, \par \smallskip

(b) $\{\sigma,\sigma'\}=\{{\rm Sc}\}$ and $\tau^J \simeq s[2,1^4]$,\par \smallskip

(c) $\{\sigma,\sigma'\}=\{{\rm St}\}$ or $\{{\rm St'}\}$ and $\tau^J \simeq s[2^3]$.\par \smallskip

This is a delicate analysis which fortunately has 
been carried out by R\"osner. Indeed, this is exactly the content of the left-bottom part of \cite[Table 4.2, p. 63]{Roesner} with $q=2$. This table shows that there are just three non-zero possible representations for $\tau^J$, denoted by $\theta_2$, $\theta_5$ and $\chi_9(1)$ there but which correspond respectively to the representations $s[2^3]$, $s[1^6]$ and $s[2,1^4]$ by R\"osner's other table \cite[Table 5.2, p. 103]{Roesner}, exactly according to the three cases above (read the table with $\Pi_1={\rm Sc}$, $\xi_\mu {\rm St} = {\rm St}'$, and take for $\mu$ either the trivial character of $\QQ_2^\times$ or its unramified quadratic character).  $\square$

\end{section}

\par \medskip

{\bf Acknowledgement.} {\rm  Ga\"etan Chenevier is supported by the C.N.R.S. and by the project ANR-14-CE25 (PerCoLaTor). He would like to thank Fabien Cl\'ery and Gerard van der Geer for inviting him to include these two appendices to their paper.}
\par \smallskip
{\sc Laboratoire de Math\'ematiques d'Orsay, Universit\'e Paris-Sud, Universit\'e Paris-Saclay, 91405 Orsay, France.}
\par \smallskip
\indent
{\it E-mail address:} {\tt gaetan.chenevier@math.cnrs.fr}
\end{appendix}

\end{document}